\documentclass[final, a4paper, 12pt]{article}

\usepackage{amsfonts}
\usepackage{amsmath}
\usepackage{amssymb}
\usepackage{amsthm}

\usepackage{bm}
\usepackage{bbm}
\usepackage[english]{babel}
\usepackage{graphicx}
\usepackage{listings}
\usepackage[mathscr]{euscript}
\usepackage{dsfont}
\usepackage[a4paper,scale=0.75]{geometry}
\usepackage{hyperref}
\usepackage{showkeys}

\newtheorem{definition}{Definition}
\newtheorem{lemma}[definition]{Lemma}
\newtheorem{theorem}[definition]{Theorem}
\newtheorem{corollary}[definition]{Corollary}
\newtheorem{proposition}[definition]{Proposition}

\newtheorem{example}[definition]{Example}
\newtheorem{remark}[definition]{Remark}

\renewcommand{\phi}{\varphi}
\renewcommand{\rho}{{\varrho}} 
\renewcommand{\epsilon}{{\varepsilon}}

\newcommand*{\N}{\ensuremath{\mathbb{N}}}

\newcommand*{\R}{\ensuremath{\mathbb{R}}}
\newcommand*{\C}{\ensuremath{\mathbb{C}}}
\renewcommand{\i}{\mathrm{i}}

\DeclareMathOperator{\diam}{diam}
\DeclareMathOperator{\supp}{supp}
\DeclareMathOperator{\dist}{dist}
\DeclareMathOperator{\Rg}{\mathcal{R}}

\renewcommand{\d}[1]{\,\mathrm{d}#1 \,}
\newcommand{\dS}[1]{\,\mathrm{dS}#1 \,}

\renewcommand{\S}{\mathbb{S}}

\newcommand{\PR}{{P_{\R^3}}}
\renewcommand{\mathcal}[1]{{\mathscr{#1}}}
\newcommand{\ch}{\mathrm{char}}
\newcommand{\conv}{\mathrm{conv}}
\newcommand{\D}{\mathcal{D}}
\newcommand{\E}{\mathcal{E}}

\newcommand{\fProp}{\Pi^{(+)}}
\newcommand{\bProp}{\Pi^{(-)}}
\newcommand{\cone}{\wedge}

\begin{document}

\title{On the determination of sets supporting unknown sources for the wave equation from radiated fields}
\author{Armin Lechleiter\thanks{Center for Industrial Mathematics, University of Bremen
; \texttt{lechleiter@math.uni-bremen.de}}}

\maketitle

\begin{abstract}
Given near or far field wave measurements generated by some unknown time- and space-dependent acoustic source, we seek to rapidly determine a domain in space-time, as small as possible, that contains the support of a source radiating these measurements. 
As for any inverse source problem, this task is challenging without further restrictions on the source, particularly due to the infinite-dimensional space of ``silent'' sources radiating zero measurements.
This first causes non-uniqueness, that is, the source in general cannot be determined uniquely, and second prevents the computation of, e.\,g., a largest set that must contain the support of the source.  
To determine small domains containing the support of some source that radiates given measurements, we exploit that solutions to the wave equation propagate along characteristics. 
We further indicate restrictions on the support of a source that allow to theoretically characterize this support uniquely from near or far field measurements. 
\end{abstract}

\section{Introduction}
This paper deals with the rapid extraction of information on the support of a time- and space-dependent source term in the wave equation from near or far field measurements of the wave  radiated (i.\,e., generated) by that source. 
It is well-known that inverse source problems generically are ill-posed and feature non-uniqueness. 
This makes it impossible to, e.\,g., uniquely determine the source from the measurements, see~\cite{Isako1990}. 
It is generally even impossible to state non-trivial upper bounds for the support of any source radiating a given measurement, as adding a non-trivial smooth function with compact support (away from the region or surface where measurements are taken) to a given solution yields a new solution that radiates the same measurement.
This motivates to determine sets that are as small as possible and support sources radiating given measurements. 
We are particularly interested in fast algorithms solving this task, as the resulting bounds can then be used as input for more accurate reconstruction techniques, e.g., to increase the computational efficiency of algorithms for time-domain inverse scattering involving time-dependent parameters. 

To fix our setting, consider scattering of an incident plane wave $v^i(t,x) = \psi(c_0 \, t- \theta \cdot x)$ with direction $\theta$ and profile $\psi \in C^\infty_0(\R)$ from a locally perturbed inhomogeneous medium described by a time-dependent sound speed $c=c(t,x): \, \R \times \R^3 \to \R$ that equals a constant $c=c_0$ outside some bounded subset of $\R \times \R^3$. 
When the incident field hits the perturbation of the background medium described by $c$, there arises a causal scattered field $v$ such that the total wave field $w = v^i+v$ satisfies the homogeneous wave equation $\ddot{w} - c^2 \Delta w =0$ in $\R\times\R^3$. 
The causal scattered field hence solves  
\[
 \ddot{v}  - c^2_0 \Delta v = f  \quad \text{in } \R \times \R^3 
 \qquad 
 \text{for } f = (c^2-c_0^2) \Delta v^i.
\]
If one is interested in numerically identifying the contrast $c^2-c_0^2$ from measurements of $v$, then any set supporting the source $f$ (i.e., containing $\supp(f)$) can be used to set up a discrete parameter space for this searched-for contrast. 
This motivates our interest in fast but not necessarily highly accurate algorithms for this task.

The latter task is particularly crucial for the design of efficient inversion algorithms in dynamic inverse problems for the wave equation, where one seeks to determine time- and space-dependent coefficients from time-dependent measurements. 
Such problems arise for instance when imaging flows via acoustic or electromagnetic waves, or in multi-static ultrasound tomography of moving or deforming objects. 
If one aims to use variational or iterative regularization schemes for parameter identification of evolving quantities, then the rapid computation of domains supporting the coefficient variation allows to reduce the dimension of the corresponding parameter space. 
This underlines the importance of fast but not highly accurate algorithms for this feature-like reconstruction problem as a possible pre-processing tool.

Inverse source problems belong to the core of the research area of inverse problems for differential equations. 
As discussed above, uniqueness for these problems can only hold under a-priori knowledge; typically, point-like sources or supports of sources defined on lower-dimensional manifolds can be uniquely identified, see, e.\,g.,~\cite{Badia2000, Isako1990}.
For time-harmonic waves, the convex scattering support of Kusiak and Sylvester~\cite{Kusia2003,Kusia2005} provides the smallest convex set that supports a source generating a time-harmonic far field measurements. 
(See also~\cite{Hanke2008} for extensions to impedance tomography.)
As the latter concept strongly relies on a unique continuation property that is missing for the wave equation, its application to the time-dependent wave equation seems difficult.  

In the Fourier domain,~\cite{Gries2012} further present an efficient approach to determine supports of sources for the time-harmonic wave equation based on the Radon transform.   
Nevertheless, transforming time-domain measurements to the Fourier domain does not always seem to be the better alternative to tackle inverse source problems, as this makes detection of moving objects impossible, prevents to exploit time-space structure of wave propagation, and requires detector measurements for (essentially) all times. 

Concerning inverse problems for the time-dependent wave equation, there are various theoretic tools to prove unique determination or stability results; exemplarily, we refer to Carleman estimates \cite{Bukhg1981, Beili2012} and the boundary control method \cite{Belis1987, Bingh2008, Oksan2013}, both of which can also be exploited numerically. 
As all these references deal with static coefficients, we also like to highlight several works for time- and space-dependent coefficients or obstacles~\cite{Stefa1989, Eskin2007, Ikeha2012, Kian2016}, all providing either uniqueness or stability results. 

We finally would like to point to a series of papers by Friedlander~\cite{Fried1962, Fried1964, Fried1967, Fried1973} on the wave equation, the far field of its solutions, and related inverse problems. 
Particular important for our results are bounds from~\cite{Fried1973} for the support of a source radiating a given far field measurement.

\textit{Notation: }
The ball of radius $R$ centered at the origin is denoted as $B_R$ and the projection of $U \subset \R\times\R^3$ onto $\R^3$ is $\PR(U) := \{ x \in \R^3: \, \exists t\in\R \text{ such that } (t,x) \in U \} \subset \R^3$.  
For an open set $U \subset \R^n$, $n\in\N$, we denote by $\D(U) = C^\infty_0(U)$ smooth test functions with compact support and by $\mathcal{E}(U) = C^\infty(U)$ smooth test functions in $U$; further,  $\D'(U)$ are distributions and $\mathcal{E}'(U)$ are distributions with compact support in $U$.
We generally extend elements of $\D(U)$ and $\E'(U)$ by zero to all of $\R^n$. 

We refer to~\cite[Chapter 4]{Rudin1991} for the usual topology on $\D(U)$ that makes this vector space a locally convex topological vector space with dual $\D'(U)$. 
We further recall that a linear map $f: \, \D(U) \to \C$ belongs to $\D'(U)$ if for all compact subsets $K \subset U$ there is $n\in\N$ and $C < \infty$ such that $|\langle f, \phi \rangle_{\D'(U) \times \D(U)}| \leq C \| \phi \|_{C^n(K)}$ for all $\phi \in \D(U)$ supported in $K$. 
We equip the vector space $\D'(U)$ with the locally convex weak-$^\ast$ topology induced by $\D(\Omega)$, i.\,e., a sequence $\{ f_n \} \in \D'(U)$ converges to $f \in \D'(U)$ if $\langle f_n-f, \phi \rangle_{\D'(U) \times \D(U)} \to 0$ for all $\phi \in \D(U)$.
Analogous definitions hold for test functions $\D(\S^2) = C^\infty(\S^2)$ and distributions $\D'(\S^2)$ on the two-dimensional unit sphere $\S^2 \subset \R^3$ or for vector-valued distributions $\D'(\R; X)$ or $\E'(\R; X)$ on the real line with values in some locally convex topological space $X$. 


\section{Expanding waves}\label{se:waves}
For a distributional source $f \in \mathcal{E}'(\R \times \R)$ with compact support we consider distributional solutions $v \in \D'(\R\times\R^3)$ to the wave equation 
\begin{equation}\label{eq:waveEqC00}
\ddot{v} - c_0^2 \, \Delta v = f \qquad \text{in } \R \times \R^3.
\end{equation}
Recall that such a solution needs to satisfy 
\begin{equation} \label{eq:waveEqDis}
  \big\langle v , \ \ddot{w} - c_0^2 \Delta w \big\rangle_{\D'(\R\times\R^3) \times \D(\R\times\R^3)} = 0
  \quad \text{for all } w \in \D(\R\times\R^3).
\end{equation}
Additionally, we require $v$ to be a causal solution, that is, $\supp(v) \subset \{ (t,x)\in \R\times\R^3: \, t \geq T_0\}$ for some $T_0=T_0(f) \in \R$.  
Causal solutions for compactly supported sources $f$ can be explicitly represented as a retarded volume potential by convolving the source with the causal fundamental solution $(t,x) \mapsto \delta(t-|x|/c_0) / (4 \pi |x|)$ of the wave equation: 
If $\supp(f) \subset \R \times B_R$ for some $R>0$, then the volume potential 
\begin{equation}\label{eq:rep0}
  v(t,x) 
  = (Vf)(t,x) 
  := \int_{B_R} \frac{f(t-|x-y|/c_0, \, y)}{4\pi |x-y|} \d{y}, \qquad (t,x) \in \R \times \R^3,  
\end{equation}
defines such a causal distributional solution. 
(We usually prefer to write this potential via an integral instead of a convolution.)  
Even if this plays no role in the sequel, we mention that~\cite{Sobol1936} shows that $v$ is the only causal distributional solution to~\eqref{eq:waveEqC00}. 
The representation formula~\eqref{eq:rep0} actually implies that the support of $v$ is even included in the set $\{ (t,x) \in \R \times \R^3: \, \big| c_0 t - |x| \big| \leq C \}$ for some constant $C=C(u) \in \R$.
Any solution to~\eqref{eq:waveEqC00} that satisfies the latter support constraint is called an expanding wave. 
By construction, retarded volume potentials with compactly supported sources always belong to this class of solutions to the wave equation. 
(We merely consider sources with compact support such that our definition is somewhat stronger than the one in~\cite{Fried1973}.)

Before stating further properties of the volume potential, note that for any $\chi \in \mathcal{D}(\R)$, the distribution $v$ defines a distribution $v_\chi = \langle v(\cdot,x),\chi \rangle_{\D'(\R)\times\D(\R)}$ in $\D'(\R^3)$ by 
\begin{equation} \label{eq:aux160}
  \langle v_\chi, \psi \rangle_{\D'(\R^3) \times \D(\R^3)} = 
  \langle v, \chi \psi \rangle_{\D'(\R\times\R^3) \times \D(\R\times\R^3)}
  \quad \text{for all $\psi \in \D(\R^3)$.}
\end{equation}
Analogously, for any $\psi \in \D(\R^3)$, the distribution $v_\psi = \langle v(t,\cdot),\psi \rangle_{\D'(\R^3)\times\D(\R^3)}$ in $\mathcal{D}'(\R)$ is defined by 
\begin{equation}\label{eq:aux168}
  \langle v_\psi, \chi \rangle_{\D'(\R) \times \D(\R)} = 
  \langle v, \chi \psi \rangle_{\D'(\R\times\R^3) \times \D(\R\times\R^3)}
  \quad \text{for all $\chi\in \D(\R)$.}
\end{equation}
For an expanding wave $v$, the representation formula~\eqref{eq:rep0} shows that $\supp(v) \subset \{ (t,x) \in \R\times \R^3: \, C_1 \leq c_0 t - |x| \leq C_2 \}$ for constants $C_{1,2}$ depending on $\supp(f)$. 
This implies that the support of $v_\chi$ and $v_\psi$ is a compact subset even if $\chi\in\E(\R)$ and if $\psi \in \E(\R^3)$, respectively.  
For an expanding wave $v$, both $v_\chi \in \mathcal{E}'(\R^3)$ and $v_\psi \in \mathcal{E}'(\R)$ hence are distributions with compact support for $\chi\in\E(\R)$ and $\psi \in \E(\R^3)$, respectively.  
In particular, the Fourier transform $\hat{v}_\psi$ of $v_\psi$ is well-defined and yields a distribution in the Schwartz class $\mathcal{S}'(\R)$, see~\cite{Rudin1991}, 
\[
  \hat{v}_\psi(k)
  := \frac{1}{\sqrt{2\pi}} \langle v_\psi, e^{-\i k \, \cdot} \rangle_{\mathcal{S}'(\R) \times \mathcal{S}(\R)} 
  \left( = \frac{1}{\sqrt{2\pi}} \int_{\R} e^{-\i kt} v_\psi(t) \d{k} \text{ if } v_\psi \in L^1(\R) \right).
\] 
Note that $\hat{v}_\psi(k)$ is analytic in $k$ as $v_\psi$ has compact support (see, again,~\cite{Rudin1991}), such that the latter equation is well-defined for fixed $k\in\R$. 
As the Fourier transform of a convolution of distributions is the product of the convolved distributions, we use~\eqref{eq:rep0},~\eqref{eq:aux160} and~\eqref{eq:aux168} to show that 
\begin{align*}
  \hat{v}_\psi(k)
  & = \frac{1}{\sqrt{2\pi}} \langle v , \psi \, e^{-\i k \, \cdot} \rangle_{\mathcal{E}'(\R; \mathcal{D}'(\R^3)) \times \mathcal{E}(\R; \D(\R^3))} 
  = \big\langle \hat{v}(k,\cdot), \,  \psi  \big\rangle_{\D'(\R^3)\times\D(\R^3)} \\
  & = \frac{1}{\sqrt{2\pi}} \bigg\langle   \int_{\Omega} \frac{\langle f(\cdot-|\cdot-y|/c_0, \, y), \, e^{-\i k\, \cdot} \, \rangle_{\mathcal{E}'(\R) \times \mathcal{E}(\R)}}{4\pi |\cdot-y|} \d{y}, \,\psi \bigg\rangle_{\D'(\R^3)\times\D(\R^3)} \\
  & = \bigg\langle \int_{\Omega} \frac{e^{\i k |\cdot -y|/c_0}}{4\pi |\cdot-y|} \hat{f}(k, y)\d{y} ,\psi \bigg\rangle_{\D'(\R^3)\times\D(\R^3)}
  \quad \text{for all $\psi \in \D(\R^3)$ and $k\in\R$.}
\end{align*}
We conclude that  
\begin{multline*}
  \langle \hat{v}, \chi \psi \rangle_{\D'(\R\times\R^3) \times \D(\R\times\R^3)}  
  = \langle v, \hat{\chi} \psi \rangle_{\D'(\R\times\R^3) \times \D(\R\times\R^3)}  
  = \langle v_\psi, \hat{\chi} \rangle_{\D'(\R) \times \D(\R)}  \\
  = \langle \hat{v}_\psi, \chi \rangle_{\D'(\R) \times \D(\R)} \nonumber 
  = \bigg\langle \int_{\Omega} \frac{e^{\i k |\cdot -y|/c_0}}{4\pi |\cdot-y|} \hat{f}(k, y)\d{y} ,\chi \psi \bigg\rangle_{\D'(\R\times\R^3)\times\D(\R\times\R^3)} 
\end{multline*}
holds for all $\chi\in \D(\R)$ and $\psi \in \D(\R^3)$, such that  
\begin{equation} \label{eq:fourier}
  \hat{v} (k,x)   = \int_{\Omega} \frac{e^{\i k |x-y|/c_0}}{4\pi |x-y|} \hat{f}(k, y)\d{y} \quad \text{holds in } \mathcal{D'}(\R \times \R^3). 
\end{equation}
We already noted that the latter distribution is analytic in $k$.
Moreover, the continuity of the Fourier transform on $\mathcal{S}'(\R)$ implies that $f \mapsto \hat{v}$ is continuous from $\E'(\R\times\R^3)$ into (vector-valued) distributions in $\mathcal{S}'(\R; \D'(\R^3))$; this subspace of $\D'(\R\times\R^3)$ contains distributions $w$ such that $w_\psi$ belongs to $\mathcal{S}'(\R)$ for all $\psi\in\D(\R^3)$. 

As $v=Vf$ solves the wave equation~\eqref{eq:waveEqC00}, $\hat{v}(k, \cdot)$ moreover solves the Helmholtz equation: 
Indeed, it is well-known that the time-harmonic volume potential $\hat{f}(k,\cdot) \mapsto  \int_{\Omega} \exp(\i k |x-y|/c_0) / [4\pi |x-y|] \, \hat{f}(k, y)\d{y}$ is for fixed $k$ continuous from $\E'(\R^3)$ into $\mathcal{D}'(\R^3)$, see~\cite{Hsiao2008}. 
Moreover, the resulting function $\hat{v}(k,\cdot)$ satisfies  
 \begin{equation}\label{eq:he}
   \Delta \hat{v}(k,\cdot) + \frac{k^2}{c_0^2} \hat{v}(k,\cdot) = \hat{f}(k,\cdot) 
   \quad \text{in } \mathcal{D}'(\R^3),
 \end{equation}
and $\hat{v}(k,\cdot)$ moreover is a real-analytic function outside the projection $\overline{\PR(\supp(f))}$ of $\supp(f)$ to $\R^3$ by Weyl's lemma, see~\cite{Horma1990}. 
The representation formula~\eqref{eq:fourier} implies that $v$ additionally satisfies Sommerfeld's radiation condition, see~\cite{Colto1998}.

\begin{lemma}\label{th:V}
  For $f \in \mathcal{E}'(\R \times \R^3)$ with compact support in $\R \times B_R$ for some $R=R(f)>0$, its retarded volume potential $v = Vf \in \mathcal{D'}(\R \times \R^3)$ is an expanding wave. 
  Further, any expanding wave can be represented by the volume potential of some $f \in \mathcal{E}'(\R \times \R^3)$. 
  The mapping $f \mapsto v$ is continuous from $\mathcal{E}'(\R \times \R^3)$ into $\D'(\R \times \R^3)$.
\end{lemma} 
See~\cite[Chapter 6]{Rudin1991} and~\cite[Section 2]{Fried1973} for a proof.

\begin{remark}\label{th:mapSobolev}
Continuity of the volume potential $V$ between suitable Sobolev spaces can also be shown by Laplace transform techniques, see~\cite{Lechl2015}:  
$V$ maps $L^2((0,T)\times B_R)$ continuously into $H^2(\R; L^2(B_R)) \cap H^1(\R; H^1(B_R))  \cap L^2(\R; H^2(B_R))$ for all balls $B_R$. 
For a square integrable source $f$, Morrey's inequality hence implies that $v = Vf$ is a continuous function in $t$ and $x$. 
The proof of Lemma 3.1 and Theorem 3.2 in~\cite{Lechl2015} can furthermore easily be adapted to show that $V$ boundedly maps sources $f \in L^2(\R; H^{-1}_0(B_R))$  into $L^2(\R; H^1(B_R))$ for all $R>0$. 
\end{remark} 

\section{Near field data, supports for sources, and unique determination}\label{se:near}
In this section, we model measurements of wave fields generated by sources on the right-hand side of the wave equation on the boundary of a $C^\infty$-smooth domain enclosing the source.  
For $T>0$ and a Lipschitz domain $\Omega \subset \R^3$ with connected complement we consider a source distribution $f \in \E'((0,T) \times \Omega)$ and identify $f$ by zero extension as an element of $\E'(\R \times \R^3)$.
We then associate to $f$ the trace $v|_{\R \times \partial\Omega}$ of the expanding wave $v \in \D'(\R \times \R^3)$ to~\eqref{eq:waveEqC00}, which is explicitly represented as a retarded volume potential in~\eqref{eq:rep0}.  
The mapping $f \mapsto v|_{\R \times \partial\Omega}$ then defines the near field operator
\[
  N: \, \E'((0,T) \times \Omega) \to \mathcal{D}'(\R\times\partial\Omega).
\]
The image space $\mathcal{D}'(\R\times\partial\Omega)$ has to be interpreted as the space of restrictions $w|_{\R\times\partial\Omega}$ of distributions $w\in\mathcal{D}'(\R\times\R^3)$.
Appendix 1 in~\cite{Shubi2001} shows that such a restriction is well-defined and bounded if we merely consider $w\in\mathcal{D}'(\R\times\R^3)$ such that the wave front set $\mathrm{WF}(w)$ does not intersect $\{ (t,x, \xi) \in \R\times\partial\Omega \times [\R^4 \setminus \{ 0 \}]: \, (t,x) \cdot \xi = 0 \}$. 
The next lemma shows that the latter condition is indeed satisfied for $v$ because $\R\times\partial\Omega$ is a time-like submanifold of $\R\times\R^3$. 

\begin{lemma}\label{th:mapN}
The near field operator $N$ is well-defined and continuous from $\E'((0,T) \times \Omega)$ into $\mathcal{D}'(\R\times\partial\Omega)$. 
For $f \in \E'((0,T) \times \Omega)$ there holds that $\supp(Nf) \subset (0,T+\diam(\Omega)/c_0) \times \partial\Omega$.
\end{lemma}
\begin{proof}
We have already showed in the previous section that $f \mapsto v$ is continuous from $\E'(\R\times\R^3)$ into $\D'(\R\times\R^3)$. 
The restriction of $v$ to any neighborhood of $\R \times \partial\Omega$ contained in $\R^3 \setminus \overline{\supp(f)}$ satisfies the wave equation with zero right-side, such that the wave front set of $v$ is contained in $\R^4 \times X_0 \setminus \{ 0 \}$, where $X_0 = \{ \xi=(\xi_{\mathrm{t}}, \xi_{\mathrm{x}}) \in \R^4: \, c_0^2 \xi_{\mathrm{t}}^2 = |\xi_{\mathrm{x}}|^2 \}$. 
As any normal to the time-like manifold $\R \times \partial\Omega$ is of the form $(0,\nu)$ for some $\nu \in \R^3 \setminus \{ 0 \}$, this manifold is non-characteristic for the wave equation.  
The above-mentioned criterion from~\cite[Appendix 1]{Shubi2001} hence shows that the restriction $v|_{\R\times\partial\Omega}$ belongs to $\mathcal{D'}(\R\times\partial\Omega)$, and that $f \mapsto v|_{\R\times\partial\Omega}$ is continuous.
%
%
The representation formula~\eqref{eq:rep0} further implies that     
\[
  \supp(v) \subset \bigg\{ (t,x)\in\R \times \R^3: \, 
  0 < t - \max_{y\in \partial \Omega} \dist(x, y)/c_0 < T \bigg\}
\]  
because the support of $f$ is compactly contained in $(0,T)\times \Omega$. 
As $x \mapsto \max_{y\in \partial \Omega} \dist(x, y)$ attains its maximum on $\partial \Omega$ by compactness, the support of the restriction $v|_{\R \times \partial\Omega} \in \D'(\R \times \partial\Omega)$ is hence included in $(0,T+\diam(\Omega)/c_0) \times \partial \Omega$. 
\end{proof}

\begin{remark}
Various other continuity properties of $N$ can be shown if that operator is defined on a Sobolev space. 
Due to Remark~\ref{th:mapSobolev} we know for instance that $N$ is continuous from both $L^2((0,T)\times \Omega)$ and $L^2(\R; H^{-1}_0(\Omega))$ into $L^2(\R; L^2(\partial\Omega))$.
\end{remark}

As for $f \in \E'((0,T) \times \Omega)$ the support of $Nf$ is compactly contained in the set $(0,T+\diam(\Omega)/c_0) \times \partial \Omega$, we can restrict the near field operator $N$ to the time interval $(0,T+\diam(\Omega)/c_0) \times \partial\Omega$ without loosing information on the sound field. 
We hence redefine 
\begin{equation}\label{eq:redefN}
    N: \, \E'((0,T) \times \Omega) \to \mathcal{E'}\big( (0,T+\diam(\Omega)/c_0) \times \partial\Omega\big).
\end{equation}

\subsection{Supports for sources}
Given the trace of a wave field in $\mathcal{E'} ( (0,T+\diam(\Omega)/c_0) \times \partial\Omega )$ due to a source supported in $[0,T] \times \Omega$, we next indicate a subset of $[0,T] \times \Omega$ that supports a source generating that wave trace.  
To this end, we define for each point $(t,x) \in \R \times \R^3$ its forward cone 
\[
  c^{(+)}(t,x) = \left\{ (\tau,z)\in \R \times \R^3: \,  c_0 [\tau-t] = |z-x|, \, \tau \geq t \right\} \subset \R \times \R^3.
\]
The explicit representation~\eqref{eq:rep0} implies for any compact set $K \subset \R\times\R^3$ that the union $\cup_{(t,x) \in K} \, c^{(+)}(t,x)$ is the smallest set containing the support of any wave generated by a source supported in $K$.

\begin{lemma}\label{th:boundNearField}
For $g \in \Rg(N)$ with $\supp(g) \subset [0,T+\diam(\Omega)/c_0] \times \partial \Omega$ there is a pre-image $f \in \E'((0,T) \times \Omega)$ supported in 
\[
  \fProp(g) := \left\{ (t,x) \in (0,T) \times \Omega): \, c^{(+)}(t,x) \cap \R \times \partial \Omega \subset \supp(g) \right\}.
\]
\end{lemma} 

The subsequent proposition proves an alternative characterization of the set $\fProp(g)$ that is simpler to approximate numerically.
Preparing this result, we define for $g \in \E'(\R \times \partial\Omega)$ and $y\in \partial \Omega$ the set $I(g,y) := \{ \tau \in \R: \,  (\tau,y) \in \supp(g) \} \subset \R$ containing all times $\tau$ such that $(\tau,y) \in \supp(g)$. 
Further,  
\[
  \bProp(g) :=  \bigcap_{y\in \partial\Omega} \bProp_y(g),
  \quad \text{where }
  \bProp_y (g) := \big\{ (t,x) \in \R\times\Omega: \,  t + |x-y|/c_0 \in I(g,y) \big\}
\]
contains all points in space-time included in the ``backward cone'' of some $(\tau,y) \in \supp(g)$. 

\begin{proposition}
For $g \in \Rg(N)$ with $\supp(g) \subset [0,T] \times \partial \Omega$ there holds 
$\fProp(g) = \bProp(g)$, such that $g$ possesses a pre-image supported in the latter set by Lemma~\ref{th:boundNearField}. 
\end{proposition}
\begin{proof}
($\subset$) If $(t,x) \in \fProp(g)$, then $c^{(+)}(t,x) \cap \R \times \partial \Omega \subset \supp(g)$. 
For each $y\in \partial \Omega$ there hence exists $t^\ast\in \R$ and $\xi =(\xi_{\mathrm{t}}, \xi_{\mathrm{x}}) \in \S^3$ with $\xi_{\mathrm{t}}>0$ and $c_0 \, \xi_{\mathrm{t}} = |\xi_{\mathrm{x}}|$, such that 
\[
  (t,x) + \alpha \xi = (t^\ast,y) \qquad \text{for some } \alpha>0. 
\]
Hence, $t+\alpha \xi_{\mathrm{t}} = t^\ast$ and $x+\alpha \xi_{\mathrm{x}} = y$, which shows that $|x-y| = \alpha |\xi_{\mathrm{x}}| = \alpha c_0 \xi_{\mathrm{t}} = c_0 [t^\ast-t]$.  
As $t^\ast \in I(g,y)$ by construction of the latter set, we conclude that 
\[
  c_0 t + |y-x| = c_0 t^\ast \in c_0 I(g,y).
\]
We conclude that $(t,x) \in \bProp_y(g)$ for all $y\in \partial\Omega$, such that $(t,x)$ belongs to $\bProp(g)$.

($\supset$) If $(t,x) \in \bProp(g)  = \bigcap_{y\in \partial\Omega} \bProp_y(g)$, then $(t,x)$ belongs to the backward cone of all points $(\tau,y) \in \supp(g)$. 
Consequently, for all $(\tau,y) \in \supp(g)$ there holds 
\[
  c_0 t + |x-y| \in c_0 I(g,y). 
\]
Thus, there is $t^\ast \in I(g,y)$ with $t^\ast > t$ such that $c_0 [t-t^\ast] + |x-y| =0$. 
As for any $y \in \partial \Omega$, the point $(t^\ast, y)$ belongs by definition of $I(g,y)$ to $\supp(g)$, we conclude that $(t,x) \in \fProp(g)$. 
\end{proof}

\subsection{Unique continuation of near field wave data}
As the solution $v = Vf$ to the wave equation~\eqref{eq:waveEqC00} does not satisfy a unique continuation principle across characteristic manifolds, trace data $v|_{\R\times\partial\Omega}$ are in general insufficient to characterize the support of the source $f$ that radiates the wave field $v$.
In the sequel we prove such a characterization under two distinct assumptions on $\supp(f)$. 
To this end, let us recall that $\conv(K)$ denotes the convex hull of a set $K \subset \R\times\R^3$. 
Obviously, linearity of the wave equation implies that we merely need to consider the case of vanishing data. 

\begin{proposition}\label{th:ucp}
Assume that $N f = 0$ on $(0,T+\diam(\Omega)/c_0) \times \partial\Omega$ for some $f \in \E'((0,T)\times \Omega)$ supported in $K = \supp(f)$. 
Then $v = Vf$ vanishes in $(\R\times\R^3) \setminus \conv(K)$. 
If $K$ additionally is convex, i.\,e., if $K=\conv(K)$, then $v$ vanishes in $(\R\times\R^3) \setminus \supp(f)$.   
\end{proposition}
\begin{proof}
(1) We have already discussed in Section~\ref{se:waves} that the Fourier transform $\hat{v}(k, \cdot)$ of the volume potential $v = Vf$ in $\D'(\R\times \R^3)$ analytic in $k\in \R$, see~\eqref{eq:fourier}. 
As $v$ solves the wave equation, $\hat{v}(k,x)$ moreover solves the Helmholtz equation 
\[
  \Delta \hat{v}(k,\cdot) + \frac{k^2}{c_0^2} \hat{v}(k,\cdot) = \hat{f}(k,\cdot) 
  \qquad \text{in } \mathcal{D'}(\R^3),
\]
together with Sommerfeld's radiation condition, which follows, e.\,g., from the representation of $v$ in~\eqref{eq:rep0}. 
Note that the support of $\hat{f}(k,\cdot) \in \D'(\R^3)$ is included in the closed set  $P_{\R^3}K = \{ x \in \R^3: \, \exists t>0 \text{ such that } (t,x) \in K\} \subset \R^3$, such that $\hat{v}(k,\cdot)$ is real-analytic in $\R^3 \setminus P_{\R^3}K$ due to Weyl's lemma. 

(2) As $v = Vf$ vanishes by assumption on $(0,T+\diam(\Omega)/c_0) \times \partial \Omega$, Lemma~\ref{th:mapN} shows that $v|_{\R\times\partial\Omega} \equiv 0$.
Thus, for all $k\in\R$ there holds that $\hat{v}(k,\cdot)$ vanishes on $\partial\Omega$ as well.  
Since $\hat{v}(k,\cdot)$ is for $k\not = 0$ a radiating solution of the Helmholtz equation in $\R^3 \setminus P_{\R^3}K$, this implies first that $\hat{v}(k,\cdot)$ vanishes in the complement of $\Omega$ and second by analyticity that $\hat{v}(k,\cdot)$ must vanish in the unbounded component of the complement of $P_{\R^3}K$ for arbitrary $k\in\R$ different from zero, see~\cite{Colto1998}. 
Continuity of $k\mapsto \hat{v}(k,\cdot)$ then implies that $\hat{v}(0,\cdot)$ vanishes in that set, too.
Thus, $v$ must also vanish in $\R \times (\R^3 \setminus P_{\R^3} K)$. 
Lemma~\ref{th:mapN} moreover show that $v$ even vanishes in $(\R \times \R^3) \setminus \big((0,T+\diam(\Omega)/c_0) \times P_{\R^3}K \big)$.
 
(3) By either Holmgren's lemma, see~\cite{Treve1975}, or by~\cite[Theorem 1]{Tarta1999}, unique continuation for solutions the the wave equation 
holds across any analytic submanifold that is non-characteristic at each of its points. 
Since $\conv(K)$ is convex, it can be approximated in the Hausdorff distance $d_H$ by a sequence of (closed) polytopes $\{ P_n \}_{n\in\N}$ that contain $\conv(K)$, see, e.\,g., the survey~\cite{Brons2008}. 
As $d_H(P_n,\conv(K)) \to 0$ as $n\to\infty$, the definition of the Hausdorff distance and the convexity of $\conv(K)$ show that the same holds for the convex hull of the $P_n$, which is a closed convex polytope. 
We can hence assume that all $P_n$ are convex polytopes and represent each of them as bounded intersection of a finite set of closed half-spaces $E(\xi_m^{(n)},a_m^{(n)}) := \{ (t,x) \in \R \times \R^3:  \, (t,x) \cdot \xi_m^{(n)} \geq a_m^{(n)} \}$ with $\xi_m^{(n)} \in \R^4$, $a_m^{(n)} \in \R$, and $m=1,\dots,M(n) \in \N$, 
\[
  P_n = \bigcap_{m=1}^{M(n)} E(\xi_m^{(n)},a_m^{(n)}). 
\]
Recall that the bounding hyperplane $\partial E(\xi_m^{(n)},a_m^{(n)})$ is characteristic for the wave equation if and only if $\xi_m^{(n)} = (\xi_{m,\mathrm{t}}^{(n)}, \xi_{m,\mathrm{x}}^{(n)})^\top  \in \R^4$ satisfies $|\xi_{m,\mathrm{x}}^{(n)}|^2 = c_0^2 [\xi_{m,\mathrm{t}}^{(n)}]^2$.

(4) Assume that some facet of $P_n$ is subset of a non-characteristic hyperplane bounding the half space $E(\xi_m^{(n)},a_m^{(n)})$, i.\,e., the direction vector $\xi_m^{(n)} = (\xi_{m,\mathrm{t}}^{(n)}, \xi_{m,\mathrm{x}}^{(n)})$ satisfies $|\xi_{m,\mathrm{x}}^{(n)}|^2 \not = c_0^2 |\xi_{m,\mathrm{t}}^{(n)}|^2$.
As $v$ vanishes in $(\R \times \R^3) \setminus ((0,T+\diam(\Omega)/c_0) \times \Omega)$, there is $a_0 \geq  a_m^{(n)}$ such that $v$ vanishes in $E(\xi_m^{(n)},a_0) \subset E(\xi_m^{(n)},a_m^{(n)})$. 
Unique continuation across $\partial E(\xi_m^{(n)},a)$ with $a_m^{(n)} \leq a \leq a_0$ then yields that $v$ vanishes in $E(\xi_m^{(n)},a_m^{(n)})$.

(5) Assume next that some facet of $P_n$ is subset of a characteristic hyperplane $\partial E(\xi_m^{(n)},a_m^{(n)})$, i.\,e., $|\xi_{m,\mathrm{x}}^{(n)}|^2 = c_0^2  |\xi_{m,\mathrm{t}}^{(n)}|^2$.
Replacing $\xi_m^{(n)}$ by $\xi_m^{(n),+} := (\xi_{m,\mathrm{t}}^{(n)} + 1/n, \xi_{m,\mathrm{x}}^{(n)})^\top$ then yields the half-space $E(\xi_m^{(n),\pm},a_m^{(n)})$. 
Without loss of generality, we can assume that $\conv(K)$ and $E(\xi_m^{(n),\pm},a_m^{(n)})$ have empty intersection. 
(Otherwise, increase the parameter $a_m^{(n)}$ until the half space does not longer contain elements of $\conv(K)$.
Since $|\xi_m^{(n),+} - \xi_m^{(n)}| \to 0$, this is always possible for some sequence $a_m^{(n),+}$ such that $0 < a_m^{(n),+} -  a_m^{(n)}  \to 0$.) 
Set now $P_n^+ = \bigcap_{m=1}^{M(n)} E(\xi_m^{(n),+},a_m^{(n)})$. 
As $|\xi_m^{(n),+} - \xi_m^{(n)}| \to 0$ and $d_H(P_n,\conv(K)) \to 0$, the definition of the Hausdorff distance implies that $d_H(P_n^+,\conv(K)) \to 0$ as well. 
Moreover, all facets of $P_n^+$ are by construction non-characteristic, such that the unique continuation argument from part (4) shows that $v$ vanishes outside $P_n^+$. 

As the polytopes $P_n^+$ approximate $\conv(K)$ in the Hausdorff distance, there is for each ball $B \subset (\R\times\R^3) \setminus \conv(K)$ a number $n_0\in \N$ such that $B \subset (\R\times\R^3) \setminus P_n^+$ for $n\geq n_0$. 
Thus, the solution $v$ hence vanishes in $(\R \times \R^3) \setminus \bigcap_{n \in \N} P_n^+ = (\R \times \R^3) \setminus \conv(K)$.
\end{proof}

The last propositions' convexity assumption on the support $K$ of the source term can actually be relaxed.
To this end, let us call $\xi = (\xi^{(1,2)}_{\mathrm{t}}, \xi^{(1,2)}_{\mathrm{x}})^\top \in \S^3 = \{ y \in \R^4: \, |y|=1\}$ with $\pm \xi_{\mathrm{t}}^2 = c_0^2 |\xi^{(1,2)}_{\mathrm{x}}|^2$ a characteristic direction in space-time. The set of all characteristic directions is $X \subset \S^3$ and two characteristic directions $\xi^{(1,2)} \in X$ are called flipped if either $\xi^{(1)}_{\mathrm{t}} = - \xi^{(2)}_{\mathrm{t}}$ or $\xi^{(1)}_{\mathrm{x}} = -\xi^{(2)}_{\mathrm{x}}$. 
For flipped characteristic directions $\xi^{(1,2)}\in X$ we denote the (closed, one-dimensional) set of directions in $\S^3$ that lie on the unique geodesic of $\S^3$ in between $\xi^{(1)}$ and $\xi^{(2)}$ by $\gamma(\xi^{(1)},\xi^{(2)})$.
These directions define the two-dimensional sector  
\[
  \cone\big(y,\xi^{(1)},\xi^{(2)}\big) = \left\{ y+\alpha \xi: \, \alpha \geq 0, \, \xi \in \gamma(\xi^{(1)},\xi^{(2)}) \right\} 
  \quad \text{for any $y \in \R^4$,} 
\]
as well as its open $\epsilon$-neighborhood $\cone_\epsilon (y,\xi^{(1)},\xi^{(2)}) \subset \R^4$ for $\epsilon>0$. 
(See Figure~\ref{fig:0} for a sketch.)
\begin{figure}[t!!!!!h!!!b!!]   
\centering
\begin{tabular}{c c}%
\includegraphics[width=0.45\linewidth]{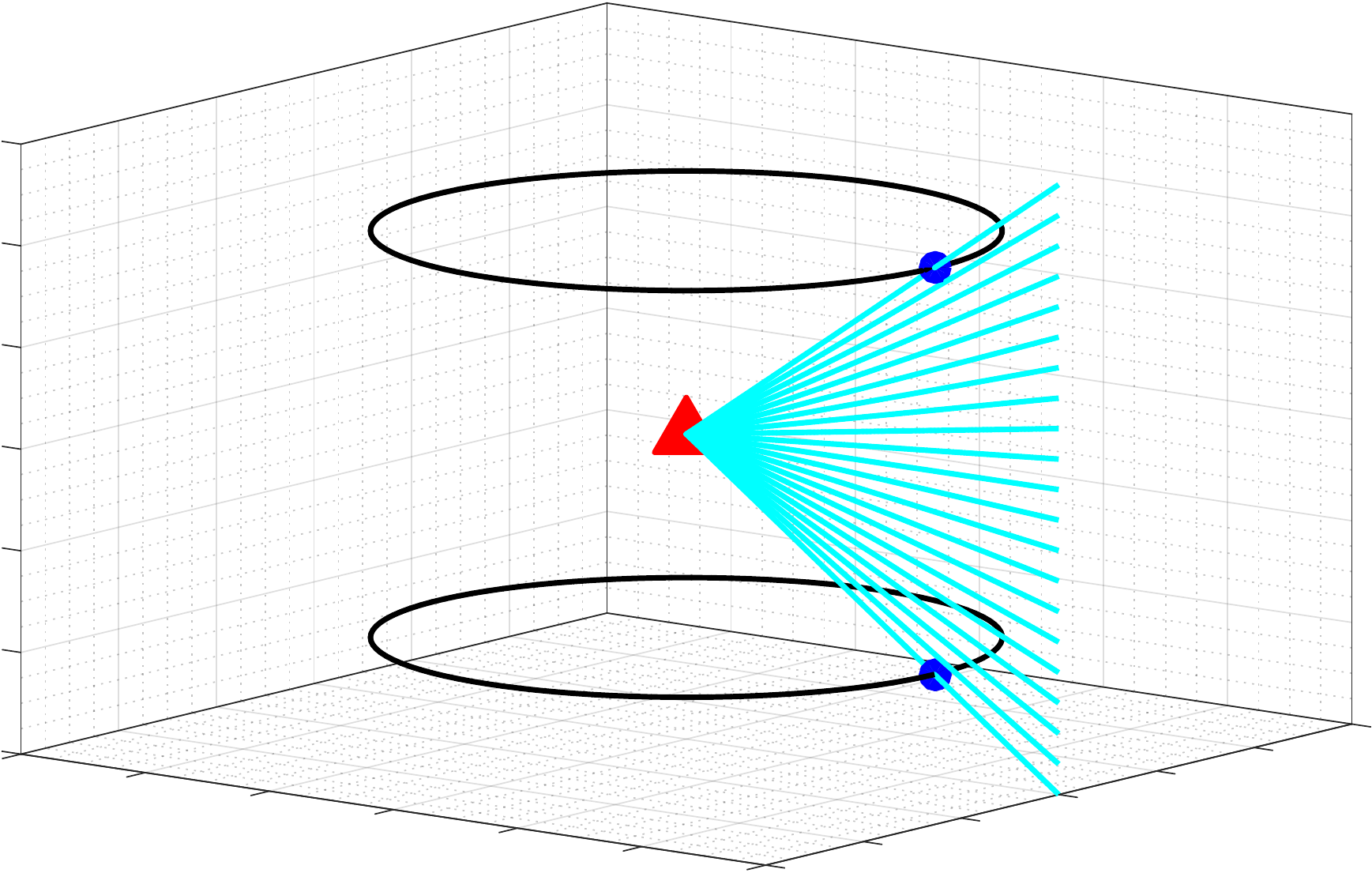}
&\includegraphics[width=0.45\linewidth]{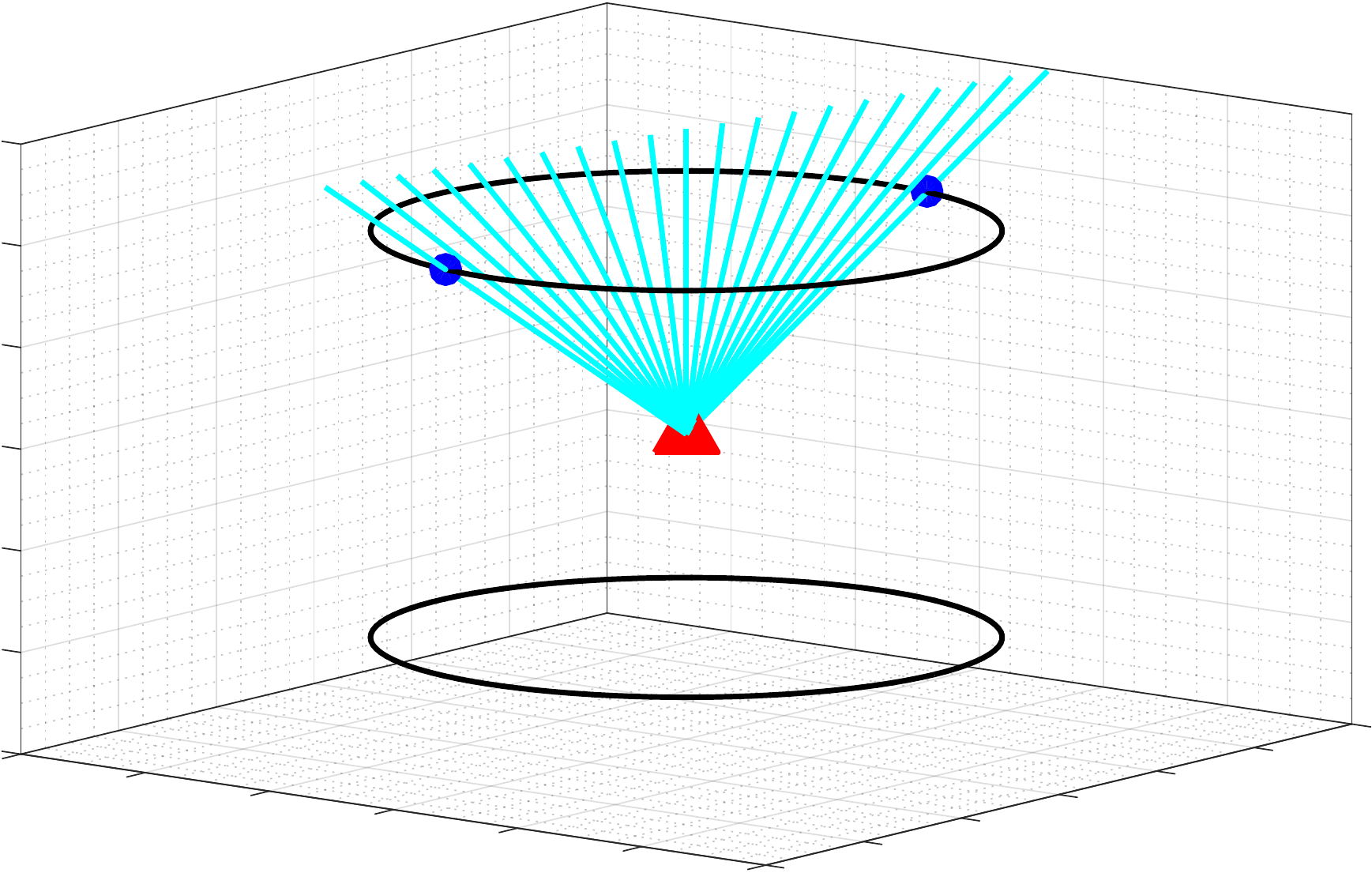}
\end{tabular}
\caption{Sketches of flipped characteristic directions $\xi^{(1,2)}$ (in three dimensions, for simplicity):
Horizontal variables correspond to $\xi_{\mathrm{x}}$, the vertical one to $\xi_{\mathrm{t}}$. 
Two circles mark the set of characteristic directions $X$, a triangle marks the origin, and two balls mark a pair of flipped characteristic directions $\xi^{(1,2)}$. 
The sector $\wedge(0,\xi^{(1)},\xi^{(2)})$ is marked by a fan of lines.}
\label{fig:0}
\end{figure} 
For any compact set $K \subset \R \times \R^3$ we finally define the characteristic hull $\ch(K)$ as complement of the union of all neighborhoods $\cone_\epsilon\big(y,\xi^{(1)},\xi^{(2)}\big)$ with flipped directions $\xi^{(1,2)}$ that do not intersect~$K$.
More precisely,
\begin{equation}\label{eq:cc}
  \ch(K) = \R^4 \setminus \bigcup \cone_\epsilon\big(y,\xi^{(1)},\xi^{(2)}\big)
\end{equation}
where the union is taken over all $y\in \R^4$, flipped directions $\xi^{(1,2)}$, and $\epsilon>0$ such that $\cone_\epsilon\big(y,\xi^{(1)},\xi^{(2)}\big) \cap K = \emptyset$. 
By construction, $\ch(K)$ hence contains $K$ (and $\overline{K}$ as well, by compactness). 
As $\ch(K)$ is the complement of a union of open sets, $\ch(K)$ moreover is closed. 

\begin{lemma}
  For any compact set $K$ there holds $K \subset \ch(K) \subset \conv(K)$.
\end{lemma}
\begin{proof}
If $y\in\R^4$ is a point outside the convex hull $\conv(K)$, then there is a closed half space $E$ such that $y$ is a boundary point of $E$ and $K \cap E = \emptyset$. 
For any such closed half space, there always exist two flipped characteristic directions $\xi^{(1,2)} \in X$ such that $\wedge(y,\xi^{(1)},\xi^{(2)})$ is a subset of $E$, which implies that $\wedge(y,\xi^{(1)},\xi^{(2)}) \cap K = \emptyset$.
As $K$ is closed, we deduce that there is $\epsilon>0$ such that $\wedge_\epsilon(y,\xi^{(1)},\xi^{(2)}) \cap K = \emptyset$, which implies that $y$ cannot belong to $\ch(K)$. 
\end{proof}

\begin{theorem}\label{th:ucp2}
Assume that $N f = 0$ on $(0,T+\diam(\Omega)/c_0) \times \partial\Omega$ for some $f \in \E'((0,T)\times \Omega)$ supported in $K = \supp(f)$. 
Then $v = Vf$ vanishes in $(\R\times\R^3) \setminus \ch(K)$.  
If there additionally holds $K = \ch(K)$, then $v$ vanishes $(\R\times\R^3) \setminus \supp(f)$. 
\end{theorem} 
\begin{proof}
It is clear by Proposition~\ref{th:ucp} that $v$ vanishes in $(\R\times\R^3)\setminus \conv(K)$. 
Consider now any point $(t,x) \not\in \ch(K)$.
Then there are two flipped characteristic directions $\xi^{(1,2)} = (\xi_{\mathrm{t}}^{(1,2)}, \xi_{\mathrm{x}}^{(1,2)})^\top \in \S^3$ and $\epsilon>0$ such that $\cone_\epsilon (y,\xi^{(1)},\xi^{(2)})$ does not intersect $K$ for some $\epsilon > 0$. 
In particular, $\cone (y,\xi^{(1)},\xi^{(2)})$ does not intersect $K$ as well. 
The proof now, roughly speaking, approximates this sector by non-characteristic surfaces. 

We distinguish two cases: 
\begin{itemize}
\item If $\xi^{(1)}_{\mathrm{t}}=\xi^{(2)}_{\mathrm{t}}$, then $\cone (y,\xi^{(1)},\xi^{(2)}) = \{ (t,x) + \alpha \xi: \, \alpha \geq 0, \, \xi \in \gamma(\xi^{(1)},\xi^{(2)}) \}$ is a sector in $\R^4$ that points towards $t=+\infty$ or $t=-\infty$.
By translating the coordinate system in $\R^4$ and rotating it with respect to the space variable, we can without loss of generality assume that $(t,x) = (0,0,0,0)^\top$ and that $\xi^{(1,2)}_{\mathrm{x}}/|\xi^{(1,2)}_{\mathrm{x}}|  = (\pm 1,0,0)^\top$.
By assumption, $\{ (0,0,0,0)^\top + \alpha \xi \}_{\alpha \geq 0} = \{ \alpha \xi \}_{\alpha \geq 0}$ hence does not intersect $K$ for any $\xi \in \gamma(\xi^{(1)},\xi^{(2)})$.

Consider for $p \in [1,2)$ and $b\geq 0$ the surfaces 
\begin{equation}\label{eq:S1} 
  \Sigma_{p,b}^{(1)} := \Big\{ (\tau,z)\in\R^4: \, c_0^p |\tau-b|^p = |z_1|^p + e^{\frac{(p-1)^2}{2-p}} [|z_2|^p + |z_3|^p] \text{ and } \tau \geq b \Big\}.
\end{equation}
This surface is characterized by $\phi^{(1)}: \R^4 \to \R$, $\phi^{(1)}(\tau,z) = c_0^p |\tau-b|^p - |z_1|^p - e^{(p-1)^2/(2-p)} [|z_2|^p + |z_3|^p]$, with (four-dimensional) gradient
\begin{multline*}
  \nabla_{t,x} \phi^{(1)} (\tau,z) 
  = p \Big(c_0^p |\tau-b|^{p-1} \mathrm{sgn}(\tau-b), \, -|z_1|^{p-1} \mathrm{sgn}(z_1), \\ 
  -e^{\frac{(p-1)^2}{2-p}} |z_2|^{p-1} \mathrm{sgn}(z_2), \, -e^{\frac{(p-1)^2}{2-p}} |z_3|^{p-1} \mathrm{sgn}(z_3) \Big)^\top. 
\end{multline*}
To show that $\Sigma_{p,b}^{(1)}$ is everywhere non-characteristic for $1<p<2$, we hence need to show that for all $(\tau,z) \in \Sigma_{p,b}^{(1)}$ there holds $(\partial_t \phi)^2 \not = c_0^2 (\nabla_x \phi)^2$, that is, 
\[
  c_0^{2p} |\tau-b|^{2p-2} 
  \not =
  c_0^2 \left[ |z_1|^{2p-2} + e^{\frac{2(p-1)^2}{2-p}} \big[|z_2|^{2p-2} + |z_3|^{2p-2} \big]\right].
\]
The latter relation is indeed satisfied for $1<p<2$, as the bounds $e^{p(p-1)/(4-2p)} - e^{(p-1)^2/(2-p)}> 0$ (since $(p-1)^2/(2-p)>0$ and $p/(2p-2) > 1$) and $3^{-1+2/p}>1$ imply that  
\begin{align*}
  |z_1|^{2p-2} + & e^{\frac{2(p-1)^2}{2-p}} [|z_2|^{2p-2} + |z_3|^{2p-2} ] 
  \geq 3^{-1+2/p} \Big[ |z_1|^{p} + e^{\frac{p(p-1)}{2(2-p)}} [|z_2|^{p} + |z_3|^{p}] \Big]^{2-2/p} \\
  & \stackrel{\eqref{eq:S1}}{=} 3^{-1+2/p} \Big[ c_0^p |\tau-b|^p + \Big(e^{\frac{p(p-1)}{4-2p}}-e^{\frac{(p-1)^2}{2-p}} \Big) [|z_2|^{p} + |z_3|^{p}] \Big]^{2-2/p} \\ 
  & > 3^{-1+2/p} \big[ c_0^p |\tau-b|^p \big]^{2-2/p}
   > c_0^{2p-2} |\tau-b|^{2p-2}. 
\end{align*}
As $b\to 0$ and $1<p \to 1$, we further note that $\Sigma_{p,b}^{(1)}$ tends in the Hausdorff distance to the characteristic one-dimensional set 
\[
  \Sigma_{1,0}^{(1)} = \{ (\tau,z)\in\R^4: \, c_0 |\tau| =  |z_1| , \, z_2=z_3 = 0 \text{ and } \tau \geq 0 \}, 
\]
which is the boundary of the closed sector $\cone (0,\xi^{(1)},\xi^{(2)}) = \{ \alpha \xi: \, \alpha \geq 0, \, \xi \in \gamma(\xi^{(1)},\xi^{(2)}) \}$.
By definition of $\ch(K)$ we conclude that $\Sigma_{1,0}^{(1)} \cap K = \emptyset$. 
As $K$ is compact and as $\Sigma_{1,0}^{(1)}$ is closed, there are hence $p_0>1$ and $b_0>0$ such that $\Sigma_{p,b}^{(1)} \cap K = \emptyset$, too, for all $p\in (1,p_0)$ and $b \in (0,b_0)$. 
This implies that any shifted surfaces $\Sigma_{p,b+s}^{(1)}$ for $p\in (1,p_0), \, b \in (0,b_0)$ and $s\geq 0$ also do not intersect $K$.
Since $v$ vanishes outside the convex hull of $K$, and as all surfaces $\Sigma_{p,b}^{(1)}$ are everywhere non-characteristic, we conclude that $v$ vanishes in the union $\cup_{0<p<p_0,0<b<\infty} \Sigma_{p,b}^{(1)}$, which is the (relative, two-dimensional) interior of the two-dimensional sector $\cone(0,\xi^{(1)},\xi^{(2)})$. 

\item If $\xi^{(1)}_{\mathrm{t}}=-\xi^{(2)}_{\mathrm{t}}$, then $\{ (t,x) + \alpha \xi: \, \alpha \geq 0, \, \xi \in \gamma(\xi^{(1)},\xi^{(2)}) \}$ is a sector in $\R^4$ that points into the direction $(0,\xi^{(1,2)}_1)$. 
By shifting the coordinate system in $\R^4$ and rotating it with respect to the space variable, we can, as in the last part, assume that $(t,x) = (0,0,0,0)^\top$ and that $\xi^{(1,2)}_{\mathrm{x}} /|\xi^{(1,2)}_{\mathrm{x}}| = (1,0,0)^\top$.

Consider for $p \in [1,2)$ and $b \geq 0$ the surfaces 
\begin{equation}\label{eq:S2}
  \Sigma_{p,b}^{(2)} := \Big\{ (\tau,z)\in\R^4: \, c_0^p |\tau|^p = |z_1-b|^p + e^{\frac{(p-1)^2}{2-p}} [|z_2|^p + |z_3|^p] \text{ and } z_1 \geq b \Big\}.
\end{equation}
These surfaces are characterized by 
\[
  \phi^{(2)}(\tau,z) = c_0^p |\tau|^p - \Big[ |z_1-b|^p + e^{\frac{(p-1)^2}{2-p}} [|z_2|^p + |z_3|^p] \Big],  
\] 
with (four-dimensional) gradient 
\begin{multline*} 
\nabla_{t,x} \phi^{(2)} (\tau,z) = p \Big(c_0^p |\tau|^{p-1} \mathrm{sgn}(\tau-t), \, -|z_1-b|^{p-1} \mathrm{sgn}(z_1-b), \, \\
 -e^{\frac{(p-1)^2}{2-p}} |z_2|^{p-1} \mathrm{sgn}(z_2), \, -e^{\frac{(p-1)^2}{2-p}} |z_3|^{p-1} \mathrm{sgn}(z_3) \Big)^\top.
\end{multline*}
As in the first part, we show that $\Sigma_{p,b}$ is everywhere non-characteristic for $1<p<2$ by proving that for all $(\tau,z) \in \Sigma_{p,b}^{(2)}$ there holds 
\[
  c_0^{2p} |\tau|^{2p-2} 
  \not =
  c_0^{2} \left[| z_1-b|^{2p-2} + e^{\frac{2(p-1)^2}{2-p}} \big[|z_2|^{2p-2} + |z_3|^{2p-2} \big] \right].
\]
Indeed, for $p\in (1,2)$ the inequalities $e^{p(p-1)/(4-2p)} - e^{(p-1)^2/(2-p)} > 0$ (as $(p-1)^2/(2-p)>0$ and $p/(2p-2)>1$) and $3^{-1+2/p}>1$ imply for $(\tau,z) \in \Sigma_{p,b}^{(2)}$ that  
\begin{align}
  |z_1-b |^{2p-2} &+  e^{\frac{2(p-1)^2}{2-p}} \big[|z_2|^{2p-2} + |z_3|^{2p-2}\big] \nonumber \\
  & \geq 3^{-1+2/p} \left[ |z_1-b|^p + e^{\frac{p(p-1)}{4-2p}} [|z_2|^p + |z_3|^p]  \right]^{2-2/p} \nonumber \\
  & \stackrel{\eqref{eq:S2}}{=} 3^{-1+2/p} \left[ c_0^p |\tau|^p +  \Big(e^{\frac{p(p-1)}{4-2p}} - e^{\frac{(p-1)^2}{{2-p}}} \Big) [|z_2|^p + |z_3|^p]\right]^{2-2/p}\nonumber\\
  & > 3^{-1+2/p} \left[ c_0^p |\tau|^p \right]^{(2p-2)/p} > c_0^{2p-2} |\tau|^{2p-2}, \label{eq:aux317}
\end{align}
such that $\Sigma_{p,b}^{(2)}$ is non-characteristic. 
For $1<p \to 1$ and $b\to 0$, there further holds that $\Sigma_{p,b}^{(2)}$ tends in the Hausdorff distance to the one-dimensional set
\[
  \Sigma_{1,0}^{(2)} = \{ (\tau,z)\in\R^4: \, c_0 |\tau| = |z_1|, \ z_2 = z_3 = 0 \text{ and } z_1 \geq 0 \}. 
\]
Again, $\Sigma_{1,0}^{(2)}$ is the boundary of the closed sector $\cone(0,\xi^{(1)},\xi^{(2)}) = \{ \alpha \xi: \, \alpha \geq 0, \, \xi \in \gamma(\xi^{(1)},\xi^{(2)}) \}$ that is by assumption outside $\ch(K)$. 
Thus, $\Sigma_{1,0}^{(2)} \cap K = \emptyset$. 
As $K$ is compact and as $\Sigma_{1,0}^{(1)}$ is closed, there are hence $p_0>1$ and $b_0>0$ such that $\Sigma_{p,b} \cap K = \emptyset$, too, for all $p\in (0,p_0)$ and $b \in (0,b_0)$. 
This implies that any shifted surface $\Sigma_{p,b}$ for $p\in (0,p_0)$ and $b \geq 0$ also does not intersect $K$.
Since $v$ vanishes outside the convex hull of $K$, and as all surfaces $\Sigma_{p,b}$ are everywhere non-characteristic, we conclude that $v$ vanishes in the union $\cup_{0<p<p_0,0<b<\infty} \Sigma_{p,b}$. 
This union is the relative two-dimensional interior of $\cone(0,\xi^{(1)},\xi^{(2)})$.
\end{itemize}
We have hence shown that for each point $(t,x)$ in the open set $\R^4 \setminus \ch(K)$ the wave field $v$ vanishes in the relative interior of the two-dimensional sector 
\[
  \cone((t,x),\xi^{(1)},\xi^{(2)}) 
  = \big\{ (t,x)+\alpha \xi: \, \alpha \geq 0, \, \xi \in \gamma(\xi^{(1)},\xi^{(2)}) \big\}
\]
from the definition of $\ch(K)$ in~\eqref{eq:cc}. 
As $\R^4 \setminus \ch(K)$ is open, we conclude that $v$ vanishes in all of $\R^4 \setminus \ch(K)$. 
\end{proof}

\begin{corollary}\label{th:cor1}
Assume that $N f_1 = N f_2$ on $(0,T+\diam(\Omega)/c_0)$ for two sources $f_{1,2} \in \E'((0,T) \times \Omega)$ such. 
Then the associated waves $v_{1,2} =Vf_{1,2} \in \D'(\R \times \R^3)$ satisfy $v_1=v_2$ in $\R \times \R^3 \setminus \ch(\supp(f_1-f_2))$.
\end{corollary} 

As any argument based on unique continuation, the proofs in this subsection fail to be constructive and do not explicitly motivate (fast) algorithms computing, e.\,g., the set $\ch(\supp(f))$ from near field data of $v=Vf$. 
(Of course, the well-known Picard's criterion might be used to distinguish whether a source in the closure of the range of $N$ effectively belongs to that range, but this requires in practice to compute a singular value decomposition of a discretization of $N$, which is not a fast procedure.)  

\section{Far field data and supports for sources}\label{se:far}
In this section we consider far field data of the time-dependent waves generated by distributional sources and show how to determine a set that supports a source generating a given far field pattern and is additionally, roughly speaking, minimal.  
We already know that a source distribution $f  \in \E'(\R\times\R^3)$ with compact support leads to a causal distributional solution to the source problem 
\begin{equation}\label{eq:waveEqC0-2}
\begin{split}
  & \ddot{v} - c_0^2 \Delta v = f \qquad \text{in } \R \times \R^3
\end{split}
\end{equation}
that vanishes in $\{ (t,x)\in \R\times\R^3: \, t<T\}$ for some $T=T(f) \in \R$ and has an explicit representation in form of a retarded volume potential, see~\eqref{eq:rep0}: 
If $\supp(f) \subset \R \times B_R$ for some $R>0$, then 
\begin{equation}\label{eq:rep-2}
  v(t,x) = \int_{B_R} \frac{f(t-|x-y|/c_0, \, y)}{4\pi |x-y|} \d{y}, \qquad (t,x) \in \R \times \R^3.
\end{equation}
At points $x = r \hat{x}$ with $R<r \to\infty$ and $\hat{x} \in \S^2=\{ |x|=1\}$, $v$ formally behaves as 
\begin{align*}
  v(\tau +r/c_0, r\hat{x})
  & = \int_{B_R} \frac{f(\tau + r/c_0 - |r \hat{x}-y|/c_0, \, y)}{4\pi |r\hat{x}-y|} \d{y} \\
  & = \int_{B_R} \frac{f(\tau-\hat{x}\cdot y/c_0, \, y)}{4\pi (r - \hat{x}\cdot y)} \left(1+ \mathcal{O}\left( \frac{1}{r}\right)\right) \d{y} \\
  & = \frac{1}{4\pi \, r} \int_{B_R} \frac{f(\tau + \hat{x}\cdot y/c_0, \, y)}{1+ \hat{x}\cdot y/r} \d{y} \left(1+ \mathcal{O}\left( \frac{1}{r}\right)\right), \qquad \tau \in \R,
\end{align*}
since $\big| \, |r \hat{x}-y| - r + \hat{x} \cdot y \big| \leq 2R^2/(r-R)$. 
The rescaled limit as $r \to\infty$ of the latter expression is called the far field $v^\infty$ of $v$, 
\begin{equation}\label{eq:farField}
  \lim_{r\to\infty} \left[ 4\pi r \,  v(\tau +r/c_0, r\hat{x}) \right] 
  = \int_{B_R} f(\tau + \hat{x}\cdot y/c_0, \, y) \d{y} 
  =: v^\infty(\tau,\hat{x}),
  \quad (\tau,\hat{x}) \in \R \times \S^2.
\end{equation}
Thus, the far field $v^\infty$ at $(\tau,\hat{x})$ formally is the integral of the source $f$ over a characteristic hyperplane through $(\tau,0)$ with normal vector $(c_0,\hat{x})$. 
(The expression in~\eqref{eq:farField} is at least well-defined if $f$ is integrable on $\R^4$.)
Similar to our notation for near field data, we denote the operator mapping sources $f$ with compact support in $\R \times \R^3$ to the far field $v^\infty$ by $F: \, f \mapsto v^\infty$. 

\begin{lemma} 
  For $f \in \E'(\R \times \R^3)$ with compact support in $\R \times B_R$ for some $R=R(f)>0$, its retarded volume potential $v = Vf \in \mathcal{D'}(\R \times \R^3)$ is an expanding wave. 
  For $\hat{x} \in \S^2$, $\hat{x} \mapsto r v(\tau + r, r \hat{x})$ tends to a limit $v^\infty(\tau,\hat{x})$ in the weak-$^\ast$ topology of $\D'(\S^2)$ and defines a continuous operator 
  \[
    F: \, \E'(\R \times \R^3) \to \E'(\R; \D'(\S^2)),
    \qquad 
    f \mapsto v^\infty . 
  \]
  The mapping $\chi \mapsto v^\infty_\chi := \langle v^\infty(\cdot, \hat{x}), \chi \rangle_{\D'(\R) \times \D(\R)}$ is continuous from $\D(\R)$ into $\D(\S^2)$. 
\end{lemma} 
\begin{proof} 
The first claim has already been shown in Lemma~\ref{th:V} and convergence of $\hat{x} \mapsto r v(\tau + r, r \hat{x})$ to $v^\infty(\tau,\hat{x})$ in $\D'(\S^2)$ in shown in~\cite[Lemma 2.1]{Fried1973}. 
The last claim concerning continuity of $\chi \mapsto v^\infty(\chi)$ follows from that reference as well, since the proof of~\cite[Lemma 2.1]{Fried1973} shows that for all $\ell\in\N$ there are constants $C=C(v^\infty,\ell)>0$ and $M_\ell \in \N_0$ such that $\| v^\infty(\chi) \|_{C^\ell(\S^2)} \leq C \| \chi \|_{C^{\ell+M_\ell}(\R)}$ holds for all $\chi \in \D(\R)$. 

It remains to show continuity of $F$ from $\E'(\R \times \R^3)$ into $\E'(\R; \D'(\S^2))$. 
As usual, we will to this end show that the transpose $F^T$, characterized by $\int_\R \int_{\S^2} Ff\, g \d{\hat{x}} \hspace*{-3pt} \d{\tau} = \int_\R \int_{\R^3} f\, F^T g \d{\hat{x}} \hspace*{-3pt} \d{\tau}$ for all $f \in \D(\R\times\R^3)$ and $g \in \D(\R; \D(\S^2))$, such that 
\[
  (F^T g)(t,y) = \int_{\S^2} g(t - \hat{x} \cdot y / c_0, \hat{x}) \dS{(\hat{x})} \qquad (t,y)\in\R\times\R^3. 
\]
Smoothness of the compactly supported function $g$ allows to exchange partial derivatives in $t$ and $y$ of any order with the latter parameter integral. 
For arbitrary $\ell\in\N$ there are hence constants $C=C(\ell)>0$ and $M_\ell \in \N_0$ such that 
\[
  \| F^T g \|_{C^\ell(\R\times \R^3)} \leq C \| g \|_{C^{\ell+M_\ell}(\R\times\S^2)}
  \qquad \text{for all  } g \in \D(\R; \D(\S^2)),
\]
which shows that continuity of $F^T$ between $\D(\R; \D(\S^2))$ and $\D(\R\times\R^3)$.
Consequently, $F$ is continuous from $\D'(\R\times\R^3)$ into $\D'(\R; \D'(\S^2))$ and in particular between the subspaces $\E'(\R\times\R^3)$ and $\E'(\R; \D'(\S^2))$: 
If $\{ f_n \}_{n\in\N}$ and $f$ belong to $\D'(\R\times\R^3)$ such that $f_n \to f$ as $n\to\infty$ in $\D'(\R\times\R^3)$, then (omitting to explicitly denote the duality products)
\[
  \langle F f_n, g \rangle
  = \langle f_n, F^\ast g \rangle
  \stackrel{n\to\infty}{\longrightarrow} 
  \langle f, F^\ast g \rangle
  = \langle F f, g \rangle
  \qquad \text{for all } g \in \D(\R;\D(\S^2)). 
\]
\end{proof}



In analogy to near field data, we say that a source $f$ generates far field data $g$ if $g=Ff$, i.e., if $g = v^\infty$ for $v=Vf$. 
In analogy to the unique continuation results for near field data in Proposition~\ref{th:ucp}, Theorem~\ref{th:ucp2}, and Corollary~\ref{th:cor1}, far field wave data in general do not uniquely determine the wave field outside the support of the source generating  the measured far field. 
Such a unique determination property again requires particular assumptions on the support of the source. 

\begin{theorem}\label{th:ucpFar}
Assume that $Ff= 0$ on $\R \times \S^2$ for some $f \in \E'(\R\times\R^3)$ supported in $K = \supp(f)$. 
Then $v = Vf$ vanishes in $(\R\times\R^3) \setminus \ch(K)$.  
If, additionally, $K = \ch(K)$, then $v$ vanishes $(\R\times\R^3) \setminus \supp(f)$. 
\end{theorem} 
\begin{proof}
Choose $R>0$ such that $K \subset [-R,R] \times B_R$. 
Theorem 3.1 in~\cite{Fried1973} (see also Theorem~\ref{th:ucpFar} below) then shows that the wave $v=Vf$ with far field $v^\infty = Ff$ vanishes in $\{ (t,x) \in \R\times\R^3 : \, |x|>R \}$. 
By causality of $v$, this wave field $v(t,\cdot)$ further must vanish for times $t<-R$. 
The representation of $v$ in~\eqref{eq:rep-2} moreover shows that the support of $v$ cannot contain points $(t,x)$ with $x \in B_R$ and $t > (1 + 2/c_0) R$. 
Hence, $\supp(v)$ is contained in $[-R, (1 + 2/c_0) R] \times B_R$. 
For any smooth manifold $\Gamma\subset\R^3$ enclosing $B_R$, the restriction of $v$ to $\R\times \Gamma$ hence vanishes and an application of Theorem~\ref{th:ucp2} yields the claim.  
\end{proof}

\begin{remark}
The far field formula~\eqref{eq:farField} shows that $(Ff)(\tau,\hat{x})$ equals the three-dimensional Radon transform of the distribution $f \in \E'(\R\times\R^3)$ along the characteristic hyperplanes $\partial E(\hat{x},\tau) = \{ (t,x)^\top \cdot (1,\hat{x}/c_0)^\top = \tau \} \subset \R^4$.   
In the literature on the Radon transform there seems to be no result on whether that amount of partial data is sufficient to (partially) determine $f$. 
\end{remark}

\subsection{Source supports}
For far field data $v^\infty$ generated by an expanding wave, there is an simple and explicit criterion due to Friedlander, see~\cite{Fried1973}, whether a given far field can be generated by a source whose support is contained in the (bi-)conical set 
\begin{equation}\label{eq:KR}
 K_R := \{ (t,x) \in \R \times \R^3: \, |t| + |x|/c_0 \leq R \} .
\end{equation}

\begin{theorem}[Friedlander]\label{th:friedlander}
  Assume that $g \in \D'(\R; \D'(\S^2))$ defines a continuous map $\chi \mapsto g_\chi =: \langle g(\cdot, \hat{x}), \chi \rangle_{\D'(\R) \times \D(\R)}$ from $\D(\R)$ into $\D(\S^2)$. 
  Then $g = Ff$ for $f \in \E'(\R \times \R^3)$ with $\supp(f) \subset K_R$ for $R \geq 0$ if and only if (1) $\supp(g) \subset \{ (\tau,\hat{x})\in \R\times\S^2: \, |\tau| < R \}$ and (2) $\hat{x} \mapsto \langle g(\cdot,\hat{x}), \tau \mapsto \tau^\ell \rangle_{\E'(\R)\times\E(\R)}$ is for $\ell \in\N$ a polynomial in $\hat{x}$ of degree less than or equal to $\ell$.
\end{theorem}
\begin{remark}
Note that assumption (1) in particular implies that the potential $v$ defined by the source supported in $K_R$ vanishes in $\{ (t,x) \in \R \times \R^3:\, | t-|x|/c_0\,|>R \}$. 
\end{remark}
\begin{proof}
If $g=(Vf)^\infty$ for a distribution $f$ supported in a subset of $K_R$, the right equality in~\eqref{eq:farField} shows that 
\[
  g(\tau,\hat{x})  
  = \int_{B_R} f(\tau+\hat{x}\cdot y/c_0, \, y) \d{y},
  \qquad  (\tau, \hat{x}) \in \R \times \S^2.  
\]  
As the far field at $(\tau,\hat{x})$ (formally) is an integral of $f$ over a characteristic hyperplane through the point $(\tau,0)$, and as the boundary of the set $K_R = \{ |t| + |x|/c_0 \leq R \}$ containing the support of $f$ is (everywhere) characteristic, $\tau$ must belong to $[-R, R]$ for that $(\tau,\hat{x})$ can belong to $\supp(g)$.
Condition (1) is hence satisfied, such that $g$ in particular has compact support. 
Further, condition (2) is due to
\begin{align}
  g_\chi(\hat{x}) 
  & = \langle g(\cdot, \hat{x}), \chi \rangle_{\D'(\R) \times \D(\R)}
   = \int_{\R} g(\tau, \hat{x}) \chi(\tau) \d{\tau} \nonumber \\
   & = \int_{\R} \int_{B_R} f(\tau+\hat{x}\cdot y/c_0, \, y) \d{y} \chi(\tau) \d{\tau} \nonumber \\
   & = \int_{\R} \int_{B_R} f(\tau, \, y) \d{y} \chi(\tau - \hat{x}\cdot y/c_0) \d{\tau} 
   \quad \text{for }\hat{x} \in \S^2, \, \chi \in \D(\R).   \label{eq:aux409}
\end{align}
(Strictly speaking, we would have to replace integrals by a suitable dual evaluation; further, the last equality merely holds in the distributional sense, i.e., after evaluation against a test function.)
As $g$ has compact support, the latter identity extends to all $\chi\in\E(\R)$ and in particular implies that $\langle g(\cdot,\hat{x}), \tau \mapsto \tau^\ell \rangle_{\E'(\R)\times\E(\R)} = \int_{\R} \int_{B_R} f(\tau, \, y) \d{y} (\tau + \hat{x}\cdot y/c_0)^\ell \d{\tau}$ is a polynomial in $\hat{x}$ of degree less than or equal to $\ell \in \N$. 

If $g$ satisfies conditions (1) and (2), then Theorem 4.1 in~\cite{Fried1973} shows that $g=Ff$ for some $f\in\E'(\R\times\R^3)$ supported in $K_R$.
\end{proof}

%
 

\begin{theorem}\label{th:shift}
  Assume that $g \in \D'(\R; \D'(\S^2))$ defines a continuous map $\chi \mapsto g_\chi =: \langle g(\cdot, \hat{x}), \chi \rangle_{\D'(\R) \times \D(\R)}$ from $\D(\R)$ into $\D(\S^2)$ and choose $z_0 \in \R^3$, $\tau_0 \in \R$, and $R \geq 0$. 
  Then $g$ is the far field of a volume potential with distributional source supported in 
  \begin{equation}\label{eq:defK}
    K_{R,\tau_0,z_0} := \{ (t,x) \in \R \times \R^3: \, |t-\tau_0| + |x-z_0| / c_0 \leq R \}
  \end{equation} 
  if and only if the shifted distribution $(\tau,\hat{x}) \mapsto g_{z_0,\tau_0}(\tau,\hat{x}) := g(\tau + \tau_0 - \hat{x}\cdot z_0/c_0,\hat{x})$ satisfies conditions (1) and (2) of Theorem~\ref{th:friedlander}. 
\end{theorem}
\begin{proof}
As $\chi \mapsto \langle v^\infty(\cdot, \hat{x}), \chi \rangle_{\D'(\R) \times \D(\R)}$ is continuous from $\D(\R)$ into $\D(\S^2)$, the definition of $g_{z_0,\tau_0}$ implies that the analogous map for that distribution is continuous as well. 
If $(\tau,\hat{x}) \mapsto g_{z_0,\tau_0}(\tau,\hat{x})$ satisfies conditions (1) and (2) of Theorem~\ref{th:friedlander}, then by that result there is $f_{z_0,\tau_0} \in \E'(\R \times \R^3)$ supported in $K_{R}$ such that $g_{z_0,\tau_0} = F f_{z_0,\tau_0}$, that is, $g_{z_0,\tau_0} = v_{z_0,\tau_0}^\infty$ for $v_{z_0,\tau_0} = V f_{z_0,\tau_0}$.  
Further, 
\begin{align}
  g(\tau,\hat{x}) 
  & = g_{z_0,\tau_0}(\tau - \tau_0 + \hat{x}\cdot z_0/c_0,\hat{x}) 
  = v_{z_0,\tau_0}^\infty(\tau - \tau_0 + \hat{x}\cdot z_0/c_0,\hat{x}) \nonumber \\
  & = \int_{B_R} f_{z_0,\tau_0}(\tau - \tau_0 + \hat{x}\cdot (y+z_0)/c_0, \, y) \d{y} \label{eq:aux480}\\
  & = \int_{|y-z_0|<R} f_{z_0,\tau_0}(\tau - \tau_0 +\hat{x}\cdot y/c_0, \, y-z_0) \d{y} \quad \text{for } (\tau,\hat{x}) \in \R\times\S^2. \nonumber
\end{align}
If we set $f(\tau,y) := f_{z_0,\tau_0}(\tau - \tau_0 + \hat{x}\cdot y/c_0, \, y-z_0)$ then the support of $f$ equals the support of $f_{z_0,\tau_0}$ shifted by $\tau_0$ in time and by $z_0$ in space. 
As $\supp(f_{z_0,\tau_0}) \subset K_R$ we conclude that $\supp(f) \subset K_{R,\tau_0,z_0}$. 
Further,~\eqref{eq:aux480} shows that $g = F f$ holds in $\D'(\R; \D'(\S^2))$.
The converse direction follows as in the proof of Theorem~\ref{th:friedlander}. 
\end{proof}

If $g \in \D'(\R; \D'(\S^2))$ decomposes into several distributions $g_1,\dots, g_M \in \D'(\R; \D'(\S^2))$ with pairwise disjoint support, the latter result can be somewhat refined.

\begin{theorem}\label{th:multiShift}
  Assume that $g \in \D'(\R; \D'(\S^2))$ defines a continuous map $\chi \mapsto g_\chi =: \langle g(\cdot, \hat{x}), \chi \rangle_{\D'(\R) \times \D(\R)}$ from $\D(\R)$ into $\D(\S^2)$.
  Further, assume that there are numbers $\tau_1 < \tau_2 < \dots < \tau_M$ and $R_1, \dots, R_M \in \R_{\geq 0}$ such that 
\begin{equation}\label{eq:WSCC}
  \supp(g) \subset \bigcup_{m=1}^M \bigg( \tau_m - \frac{R_m}{c_0}, \, \tau_m + \frac{R_m}{c_0} \bigg) \times \S^2
  \quad \text{and} \quad 
  \tau_{m+1} - \tau_m > \frac{1}{c_0}\big(R_{m+1} + R_m \big)
\end{equation}
and define $g_m \in \D'(\R; \D'(\S^2))$ by $g_j = g \, \mathbbm{1}_{(\tau_m - R_m/c_0, \, \tau_m + R_m/c_0) \times \S^2}$ for $m=1,\dots,M$.\\[1mm]
(a) Then $g$ is the far field of a retarded volume potential for some distributional source supported in a subset of  $\bigcup_{m=1}^M K_{R_m,\tau_m,0}$, see~\eqref{eq:defK}.\\[1mm]
(b) Assume that there are pairs $(\tau_m', z_m') \in K_{R_m,\tau_m,0}$ and numbers $R_m' \in [0,R_m)$ with $K_{R_m',\tau_m',z_m'} \subset K_{R_m,\tau_m,0}$ such that $(\tau,\hat{x}) \mapsto g_m(\tau + \tau_m' - \hat{x}\cdot z_m'/c_0,\hat{x})$ is supported in $[R_m',R_m']$.  
Then $g$ is the far field of a retarded volume potential for some distributional source supported in a subset of $\bigcup_{m=1}^M K_{R_m',\tau_m',z_m'}$.
\end{theorem}
\begin{remark}
By assumption~\eqref{eq:WSCC}, the sets $K_{R_m,\tau_m,0}$ from (a) are pairwise disjoint.
\end{remark}
\begin{proof}
Let us first note that~\eqref{eq:WSCC} implies that the support of $g$ decomposes into $M$ disjoint parts included in the open sets $( \tau_m - R_m/c_0, \, \tau_m + R_m/c_0 \big) \times \S^2$. 
Consequently, each $g_m$ is indeed a distribution with compact support and the sum of these elements of $\E'(\R; \D'(\S^2))$ equals $g$.
 
As the pairwise intersections of the supports of $g_1, \dots, g_M$ are empty, there is for all $m=1,\dots,M$ a cut-off function $\zeta_m \in \D(\R)$ such that $\zeta_m \equiv 1$ on $[ \tau_m - R_m/c_0, \, \tau_m + R_m/c_0]$ and $\zeta_m \equiv 0$ on $\bigcup_{j\not=m} \big[ \tau_j - R_j/c_0, \, \tau_j + R_j/c_0 \big]$. 
Thus, $g_m(\chi) = g_m(\chi \, \zeta_m) = g(\chi \, \zeta_m)$ for all $\chi \in \D(\R)$ and the continuity of $\chi \mapsto g_\chi$ from $\D(\R)$ into $\D(\S^2)$, together with the product rule, shows that $\chi \mapsto g_m(\chi)$ is continuous from $\D(\R)$ into $\D(\S^2)$ as well. 

Theorem~\ref{th:friedlander} applied to each $g_m$ now implies that $g_m = (Vf_m)^\infty$ for some $f_m \in \D'(\R\times \R^3)$ supported in $K_{R_m,\tau_m,0}$. 
This shows part (a), because $g=v^\infty$ for $v = \sum_{m=1}^M V f_m$. 

To prove part (b), assume that for all $m=1,\dots,M$ there is a point $(\tau_m',z_m') \in K_{R_m,\tau_m,0}$ and a radius $R_m'$ with $K_{R_m',\tau_m',z_m'} \subset K_{R_m,\tau_m,0}$ such that additionally the support of $(\tau,\hat{x}) \mapsto g_m(\tau + \tau_m' - \hat{x}\cdot z_m/c_0,\hat{x})$ is supported in $[\tau_m'-R_m/c_0,\tau_m'+R_m/c_0]\times \S^2$. 
Theorem~\ref{th:shift} then implies that $g_m$ is far field of some retarded volume potential supported in $K_{R_m',\tau_m',z_m'}$, which shows the claim of (b). 
\end{proof}

\begin{example}
We illustrate Theorems~\ref{th:friedlander},~\ref{th:shift}, and~\ref{th:multiShift} by examples for point sources.\\[1mm]
(a)
The Dirac distribution $\delta_{(\tau_0,z_0)}$ at $(\tau_0,z_0) \in \R \times \R^3$ is a source that radiates the field $v(t,x) = \delta_0(t-\tau_0-|x-z_0|/c_0)/|x-y|$ where $\delta_0$ is the scalar Dirac distribution at the origin. 
By~\eqref{eq:farField}, the far field of $v$ equals $v^\infty(\tau,\hat{x}) = \delta_0(\tau -\tau_0 + \hat{x} \cdot z_0/c_0)$
and is supported in $[\tau_0 - |z_0|/c_0, \tau_0 + |z_0|/c_0] \times \S^2$. 
Integrating $v^\infty$ against $\tau \mapsto \tau^\ell$ for $\ell\in\N$ shows that 
\[
  \hat{x} \mapsto \int_\R \delta_0(\tau -\tau_0 + \hat{x} \cdot z_0/c_0) \tau^\ell \d{\tau} = (\tau_0 - \hat{x} \cdot z_0/c_0)^\ell, \qquad \hat{x}\in\S^2,
\]
further is a polynomial of degree at most $\ell$. 
Theorem~\ref{th:friedlander} hence states that $K_{\max |\tau_0 \pm |z_0|\,| } = \{ |t - |x|/c_0 \,| \leq \max(|\tau_0 \pm |z_0|\,|) \}$ contains a source radiating $v^\infty$. 
Shifting the far field in $\tau_0$ and $z_0$ in order to reduce the volume of the supporting set, by Theorem~\ref{th:shift} we find that $K_{0,\tau_0,z_0} = \{ (\tau_0,z_0) \}$ supports $v^\infty$, since the shifted far field $v^\infty_{z_0,\tau_0}(\tau,\hat{x}) := v^\infty(\tau + \tau_0 - \hat{x} \cdot z_0/c_0,\hat{x}) = \delta_0(\tau)$ has support $\{ 0 \} \times \S^2$.\\[1mm] 
(b) 
For two different source points $(\tau_{1,2},z_{1,2}) \in \R \times \R^3$, the source $f= \delta_{(\tau_1,z_1)} + \delta_{(\tau_2,z_2)}$ radiates the far field $v^\infty(\tau,\hat{x}) = [\delta(\tau -\tau_1 + \hat{x} \cdot z_1/c_0)+\delta(\tau -\tau_2 + \hat{x} \cdot z_2/c_0)]$ supported in $\bigcup_{j=1,2} [\tau_j - |z_j|/c_0, \tau_j + |z_j|/c_0] \times \S^2$. 
Theorem~\ref{th:friedlander} hence shows that $K_{\max_{j=1,2}|\tau_j \pm |z_j|\,|}$ contains a source that radiates $v^\infty$.

Theorem~\ref{th:shift} allows to reduce this set to any $K_{R,\tau,z}$ that contains both points $(\tau_{1,2},z_{1,2})$ in its boundary (which is satisfied, e.\,g., by choosing $z=z_1$, $\tau=1/2 \, (\tau_1+\tau_2+|z_2|-|z_1|)$, and $R = c_0/2 (\tau_2-\tau_1+|z_2|-|z_1|)$).

Finally, $v^\infty$ can be split into two far fields via condition~\eqref{eq:WSCC} if either $\max [\tau_1 \pm |z_j|/c_0] < \min[\tau_2 \pm |z_2|/c_0]$ or $\max [\tau_2 \pm |z_2|/c_0] < \min[\tau_1 \pm |z_1|/c_0]$. 
Under this condition, we can hence split $v^\infty = v^\infty_1 + v^\infty_2$ with $v^\infty_{1,2}(\tau,\hat{x}) = \delta(\tau -\tau_{1,2} + \hat{x} \cdot z_{1,2}/c_0)$. 
Theorem~\ref{th:multiShift}(a) then states that $K_{|z_1|,\tau_1,0} \cup K_{|z_2|,\tau_2,0}$ supports $v^\infty$. 
Theorem~\ref{th:multiShift}(b) next considers the supports of 
\[
  (\tau,\hat{x}) \mapsto v^\infty_{1,2}(\tau + \tau' - \hat{x}\cdot z'/c_0,\hat{x})
  = \delta_0(\tau + (\tau'-\tau_{1,2}) + \hat{x}\cdot (z_{1,2}-z')/c_0),
\]
which reduce to the minimal set $\{ 0\} \times \S^2$ if and only if $\tau'=\tau_{1,2}$ and $z' = z_{1,2}$. 
As $(\tau_{1,2},z_{1,2}) \in K_{|z_{1,2}|,\tau_{1,2},0}$, Theorem~\ref{th:multiShift}(b) allows to conclude that $\{ (\tau_1,z_1), (\tau_2,z_2) \}$ is a set that supports a source generating $v^\infty$, which obviously is an optimal result. 
\end{example}

\subsection{Numerical experiments}
The last Theorem~\ref{th:multiShift} can be exploited algorithmically to determine small conical sets supporting sources for given far field data. 
In the sequel, we call this algorithm the conical support algorithm (CSA):
If a far field $g: \, \R \times \S^2$ decomposes into $M$ disjoint parts $g_m$, then CSA proceeds in two steps: 
\begin{itemize} 
  \item[(I)] Determine bounds $K_{R_m,\tau_m,0}$ for the support of the individual sources generating $g_m$ for $m=1,\dots,M$ according to Theorem~\ref{th:friedlander}. 
  \item[(II)] For $m=1,\dots,M$ do: 
  \begin{itemize}
    \item[$\bullet$] Define a grid of test points $z'$ in $\{ |x| < R_m \}$.
    \item[$\bullet$] Set $R_m^\ast=R_m$, $\tau_m^\ast=\tau_m$, and $z_m^\ast = 0$. 
    \item[$\bullet$] For each $z'$ compute the smallest interval $[T^-_{m, z'},T^+_{m, z'}]$ such that the shifted far field $g_{m, z'}: \, (\tau,\hat{x}) \mapsto g_m(\tau - \hat{x} \cdot z', \hat{x})$ is supported in $[T^-_{m, z'}, \, T^+_{m, z'}] \times \S^2$.
    \item[$\bullet$] If $T^+_{m, z'} - T^-_{m, z'} < R_m^\ast$ set $R_m^\ast= T^+_{m, z'} - T^-_{m, z'}$, $\tau_m^\ast = (T^+_{m, z'} + T^-_{m, z'})/2$, and $z_m^\ast=z'$. 
  \end{itemize}
\end{itemize}
The $M$ triples $(R_m^\ast,\tau_m^\ast,z_m^\ast)$ returned by CSA define conical sets $K_{R_m^\ast,z_m^\ast,\tau_m^\ast}$ that contain the support of $M$ sources radiating the $M$ far fields $g_m$. 
Note that the numbers $T^+_{m, z'} - T^-_{m, z'}$ and $(T^+_{m, z'} + T^-_{m, z'})/2$ define the width and the center of the conical set $K_{T^+_{m, z'} - T^-_{m, z'}, (T^+_{m, z'} + T^-_{m, z'})/2, 0}$, respectively.
The condition $T^+_{m, z'} - T^-_{m, z'} < R_m^\ast$ in step (II) of the latter algorithm hence checks whether the latter conical set is smaller than the currently determined conical set containing the support of $g_m$.  
Of course, there is no theoretic guarantee that there is a unique smallest conical set with that property.

\begin{remark}
(a) Computing the numbers $T^\pm_{m, z'}$ in step (II) of CSA does not require the costly computation of the shifted far field $g_{m, z'}$ itself. 
Computing instead functions $T^\pm_m: \, \S^2 \to \R$ such that for each direction $\hat{x} \in \S^2$ the interval $[T^-_m(\hat{x}), \, T^+_m(\hat{x})]$ is the support of $g_m(\cdot, \hat{x}): \, \R \to \R$ allows to quickly compute $T^\pm_{m,z'}$ as  
\[
  T^+_{m,z'} = \sup_{\hat{x} \in \S^2} \big[ T^+_m(\hat{x}) + z' \cdot \hat{x} / c_0 \big] 
  \quad \text{and} \quad 
  T^-_{m,z'} = \inf_{\hat{x} \in \S^2} \big[ T^-_m(\hat{x}) + z' \cdot \hat{x} / c_0 \big]. 
\]
(b) The intersection $K_{R_m^\ast,\tau_m^\ast,z_m^\ast} \cap K_{R_m,\tau_m,0}$ of the two conical sets determined in steps (I) and (II) of CSA contains the support of $g_m$, which might provide a support bound that is a strict subset of $K_{R_m^\ast,\tau_m^\ast,z_m^\ast}$. 
\end{remark}

To illustrate feasibility and robustness of the CSA via two numerical examples, let us set $c_0 = 1$ and consider sources where the generated far field is explicitly computable. 
For points $p_1 =  (1.2, 0, 0)^\top$ and $p_2 = (-0.4, 0.4, -0.4)^\top$ and shifts $\tau_1 = -1.3$ and $\tau_2 = 2.5$, we set 
\begin{equation}\label{eq:f12}
  f_{1,2}(t,x) =
    e^{-8 (t-\tau_{1,2})^2} \, \delta_{x=p_{1,2}} 
    \qquad \text{if } (t-\tau_{1,2})^2 \leq 1 \text{ and } x \in \R^3,
\end{equation}
and $f_{1,2}(t,x) = 0$ else. 
By~\eqref{eq:farField}, these sources radiate causal waves with far fields 
\[
  v^\infty_{1,2}(\tau, \hat{x}) = 
    e^{-8 (\tau - \tau_{1,2} + \hat{x} \cdot p_{1,2})^2} 
    \qquad \text{if } (\tau - \tau_{1,2} + \hat{x} \cdot p_{1,2})^2 \leq 1 \text{ and } \hat{x} \in \S^2, 
\]
and $v^\infty_{1,2}(\tau, \hat{x}) = 0$ else. 
We further consider a time-dependent source point at position $s(t) = (2+0.3 \, \cos(2 \pi t / 4), \, 2+0.3 \, \sin(2 \pi t / 4), \, 0.3 \, \sin(2 \pi t / 4) \big)^\top$ for $0 < t < 4$ that is modeled by the source 
\begin{equation}\label{eq:f}
  f(t,x) =
    e^{-8 [ t / \sin(\pi t / 4)^{1/4} ]^2} \, \delta_{x=s(t)} 
    \qquad \text{if } 0 < t < 4 \text{ and } x \in \R^3, 
\end{equation}
and $f(t, x) = 0$ else. 
The causal wave generated by this source radiates the far field 
\[
  v^\infty(\tau, \hat{x}) =
    e^{-8 \big[ ( \tau + \hat{x} \cdot s(t) ) / \sin(\pi (\tau + \hat{x} \cdot s(t)) / 4)^{1/4} \big]^2} 
    \qquad  \text{if } 0 < \tau+\hat{x} \cdot s(t) < 4 \text{ and } \hat{x} \in \S^2,
\]
and $v^\infty(\tau, \hat{x}) = 0$ else. 
We store the evaluation of these far fields for 440 directions on the unit sphere and 1000 equally spaces time steps in the time interval $[-10, 10]$ in a far field matrix of size 440 times 1000. 
(The 440 directions on the sphere are midpoints of the quadrangles of a surface mesh of the sphere.) 

After evaluating the above expressions we add to each data matrix a multiple of a random matrix with uniformly distributed entries in $[-1,1]$ such that the relative noise level in either case equals 5 percent, measured in the spectral matrix norm. 
(We do not exploit knowledge of the noise level in the CSA!)
Figure~\ref{fig:1} shows the two noisy far field matrices for data $v^\infty_1 + v_2^\infty$ and $v^\infty$. 
\begin{figure}[h!!!t!!!!!b!!]
  \centering
\begin{tabular}{c c}
\hspace*{-1mm}\includegraphics[width=0.48\linewidth]{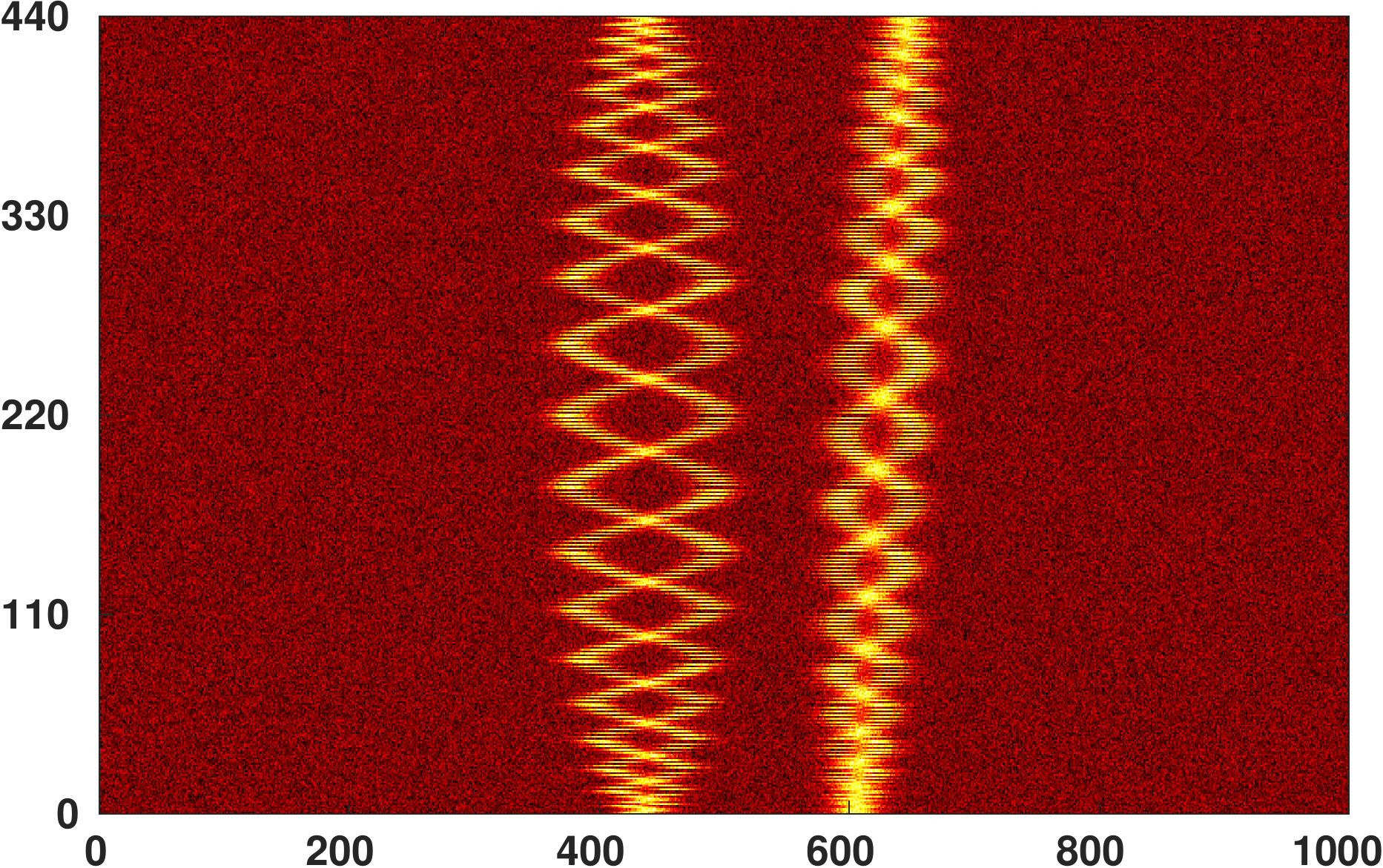}
&\includegraphics[width=0.48\linewidth]{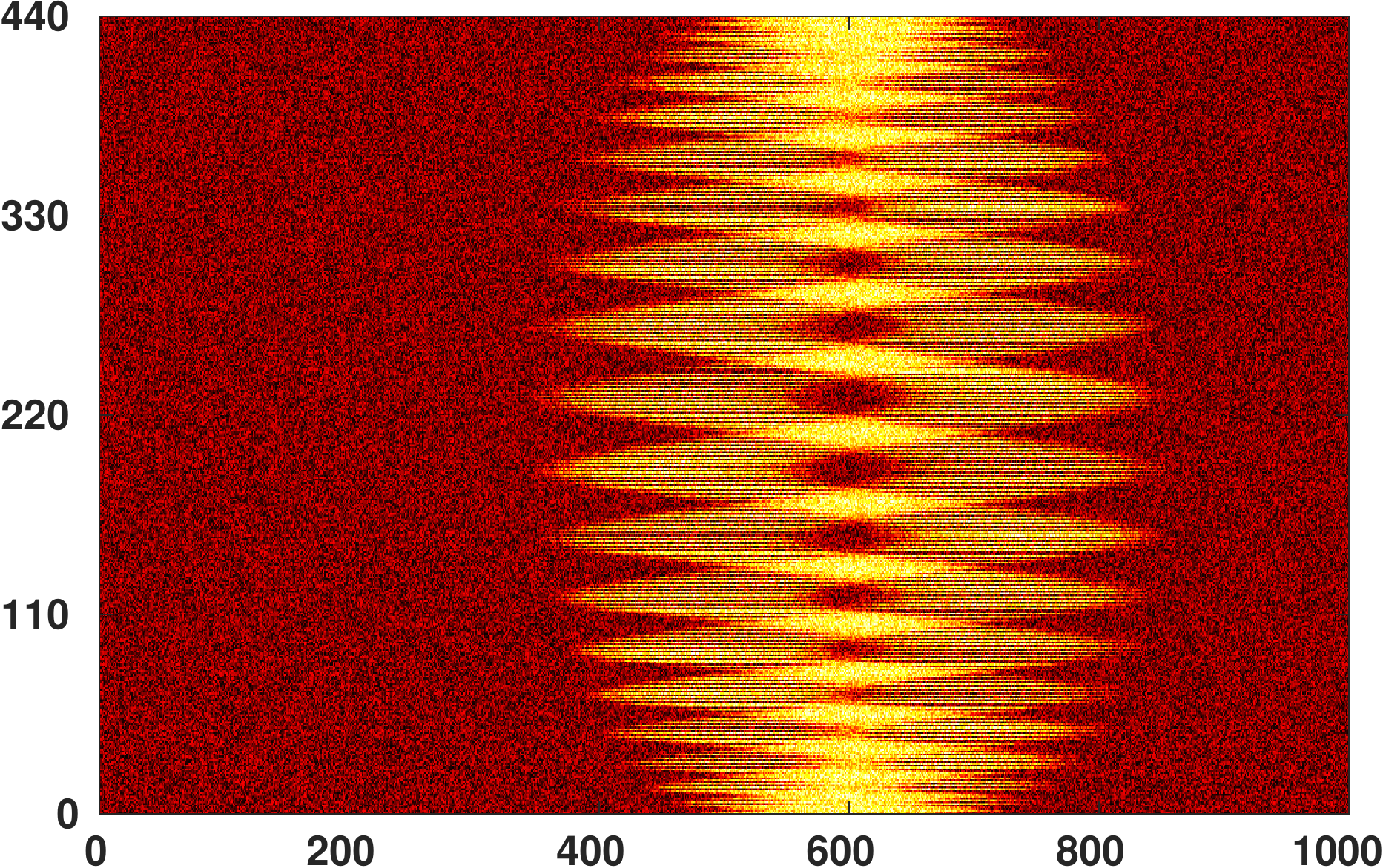}\\
(a) & (b) 
\end{tabular}
\caption{Noisy far field matrices for the far fields $v_1^\infty + v_2^\infty$ in (a) and $v^\infty$ in (b), for 1000 equally spaces points in time in [-10, 10] and $440$ directions on the unit sphere. 
The artificial noise level equals 0.05. }
\label{fig:1}
\end{figure}

Before starting the CSA, we estimate the maximal magnitude $A>0$ of noised zero entries in the data matrix for each of the 440 receiver positions individually. 
Precisely, for each sensor, we set $A$ to be 1.2 times the maximum of each of the first 20 (noisy) far field values (which, by assumption, should vanish for noise-free data). 
This estimate then serves to separate sources for different receivers by constructing connected components (constructed as index set) of the far field matrix: 
For each individual receiver, two simulated recorded values are considered as belonging to the same connected component if both are larger in magnitude than $A$ and if the time difference in between recording these two values is less than 0.08.
Further, it two receivers belong to adjacent quadrangles of the sphere's surface mesh that we use to define the receiver positions, and if two connected components for these two receivers share a common index, then these two individual connected components of the two receivers are considered as belonging to the same connected component of the entire data matrix. 
Partitioning the far field matrix into several connected components, i.e., index sets, can then be done by sequentially checking all ``adjacent'' receivers. 

When one applies this separation procedure to the first far field matrix containing evaluations of $v^\infty_1 + v^\infty_2$ generated via $f_1 + f_2$, one obtains two connected components corresponding to the two sources. 
Naturally, both of these components are inexact when compared to the support of $v^\infty_1 + v^\infty_2$ due to the added noise. 
Application of step (I) of CSA then computes that the first source $f_1$ is contained in $K_{R_1, \tau_1, 0}$ with $R_1 = 1.201$ and $\tau_1 = -1.291$.
Recall that the true source $f_1$ is a Dirac distribution at $(1.2, 0, 0)^\top$ shifted by $-1.3$ in time with a time-dependent profile given in~\eqref{eq:f12}.   
The support of $f_2$ is estimated by step (a) of CSA to be contained in $K_{R_2, \tau_2, 0}$ with $R_2 = 0.790$ and $\tau_2 = 2.492$; recall from~\eqref{eq:f12} that $f_2$ equals a Dirac distribution at $(-0.4, 0.4, -0.4)^\top$ shifted in time by $2.5$.
Note that we did not check numerically whether condition~\eqref{eq:WSCC} for the two separated far fields is satisfied. 

To refine these estimates, we define test points $z_{1,2}$ that cover $\{ |x| \leq R_{1,2} \}$ as $j/J \, R_{1,2} \hat{x}_m$ for $j=1, \dots, 51$, $J=34$, and the 440 directions $\hat{x}_m$ that we already used to simulate the far field matrices. 
(These values provided sufficiently accurate results.)
These test points are then used in step (II) of the CSA. 
For the (separated) source $f_1$ the algorithm computes that this source's support is included in $K_{R_1^\ast, \tau_1^\ast, z_1^\ast}$ with $R_1^\ast = 0.48$, $z_1^\ast = (1.18, -0.09, 0.10)^\top$, and $\tau_1^\ast =  -1.29$. 
For $f_2$, the resulting conical set $K_{R_2^\ast, \tau_2^\ast, z_2^\ast}$ including $\supp(f_2)$ is defined by  $R_2^\ast = 0.38$, $z_1^\ast = (-0.40, 0.40, -0.38)^\top$, and $\tau_2^\ast =  2.51$.
The computation time a naive implementation of CSA in MATLAB on a standard workstation without parallelization is for this example about 1.2 seconds. 
The quality of both results is reasonable, as the computed centers $(z_{1,2}^\ast, \tau_{1,2}^\ast)$ of the conical sets are less than 0.11 away from the true source points in the maximum norm, and the width of the conical sets corresponds roughly to the length of the time interval in which the time-profile $\exp(-8(t-\tau_{1,2})^2)$ of the sources $f_{1,2}$ is larger than the estimated maximal magnitude $A$ of the additive noise.

In the second numerical example we consider the perturbed far field $v^\infty$ corresponding to the source $f$ from~\eqref{eq:f}.
Figure~\ref{fig:1}(b) already indicates that this source cannot be separated into two disjoint sources (which is also due to our parameters when computing the connected components of the data matrix). 
The parameters for the conical set $K_{R_1,\tau_1,0}$ containing the support of $f$ computed in part (a) of CSA equal $R_1 = 2.73$ and $\tau_1 = 1.89$. 
The refined bound computed in part (b) is $K_{R_1^\ast,\tau_1^\ast,z_1^\ast}$ with parameters $R_1^\ast = 1.96$, $z_1^\ast = (2.04, 1.74, -0.24)^\top$, and $\tau_1^\ast = 2.01$. 
The computation time for this example is roughly 0.7 seconds. 
The quality of the resulting set is about the same as for the first example; e.g., the center $(z_1^\ast, \tau_1^\ast)$ of the resulting conical set is close to the mean value $((2,2,0)^\top, 2)$ of the moving source point $s(t)$. 
Further, step (I) of CSA returns a radius $R_1 = 2.73$ which is reduced in the second step to $R_1^\ast = 1.96$ by choosing the center $z_1^\ast$ of the cone $K_{R_1^\ast,\tau_1^\ast,z_1^\ast}$ approximately to the mean $((2,2,0)^\top, 2)$ of the source trajectory.

\providecommand{\noopsort}[1]{}

\end{document}